\documentclass[12pt,reqno]{ip-journal}

\usepackage{amsthm, amsmath, amsfonts, amscd, amssymb}
\usepackage{stmaryrd,mathrsfs}

\usepackage[breaklinks,colorlinks,citecolor=teal,linkcolor=teal,urlcolor=teal,pagebackref,hyperindex]{hyperref}
\usepackage{color}
\usepackage{tikz}
\usetikzlibrary{arrows} \usetikzlibrary{patterns}
\usepackage[all]{xy}
\usepackage{tikz-cd}
\usepackage{appendix}

\usepackage[top=2.5cm,bottom=2.5cm,left=2.6cm,right=2.6cm]{geometry}

\newtheorem{theorem}{Theorem}[section]
\newtheorem{lemma}[theorem]{Lemma}
\newtheorem{prop}[theorem]{Proposition}
\newtheorem{coro}[theorem]{Corollary}

\newtheorem{defn}[theorem]{Definition}

\theoremstyle{definition}
\newtheorem{example}[theorem]{Example}

\newtheorem{remark}[theorem]{Remark}

\numberwithin{equation}{section}

\def \fk{\mathfrak}
\def \msf{\mathsf}

\def \mc{\mathcal}

\def \inv{^{-1}}
\def \0{\infty}

\def \qq{\quad}
\def \p{\partial}

\def \rto{\rightarrow}

\def \rrto{\rightrightarrows}
\def \Hom{\mathrm{Hom}}

\def \co{\colon\thinspace}

\def \and{\qq\mathrm{and}\qq}

\def \CF{{\mc F}}
\def \CG{{\mc G}}
\def \CH{{\mc H}}

\def \CK{{\mc K}}

\def \CV{{\mc V}}
\def \CX{{\mc X}}
\def \CY{{\mc Y}}
\def \CZ{{\mc Z}}

\def \SB{{\msf B}}
\def \SBAut{{\msf{BAut}}}

\def \SE{{\msf E}}
\def \SF{{\msf F}}
\def \SG{{\msf G}}
\def \SH{{\msf H}}
\def \SK{{\msf K}}
\def \SN{\mathsf N}
\def \SP{{\msf P}}
\def \SQ{{\msf Q}}
\def \ST{{\msf \Psi}}

\def \SX{{\msf X}}
\def \SY{{\msf Y}}
\def \Y{\mathsf{Y}}
\def \SZ{{\msf Z}}

\def \f{{\msf f}}
\def \g{{\msf g}}
\def \h{{\msf h}}

\def \s{{\msf s}}

\def \u{{\msf u}}

\def \d{{\msf d}}

\def \SP{{\msf \Psi}}
\def \SAd{{\msf{Ad}}}

\def \SFib{{\msf{Fib}}}
\def \sk{\mathrm{sk}}
\def \SHom{{\mathrm{SHom}}}
\def \Kan{\mathrm{Kan}}

\def \Aut{\mathrm{Aut}}

\def \Ad{{\mathrm{Ad}}}
\def \uni{{\mathrm{uni}}}
\def \Vect{{\mathrm{Vect}}}

\def \Acyc{{\mathrm{Acyc}}}
\def \GL{\mathrm{GL}}

\def \Ad{\mathrm{Ad}}
\def \Acyc{\mathrm{Acyc}}

\def \SAut{\mathsf{Aut}}
\def \Diff{\mathbf{Diff}}

\def \bfe{\mathbf e}
\def \SHom{\mathsf{Hom}}

\def \SQ{\mathsf Q}
\def \SP{\mathsf P}
\def \SAd{\mathsf {Ad}}
\allowdisplaybreaks

\begin{document}

\title[Automorphisms of Lie groupoids and symplectic reduction on orbifolds]{Automorphisms of Lie groupoids and symplectic reduction on orbifolds}

\author{Bohui Chen}
\address{Department of Mathematics and Yangtz center of Mathematics, Sichuan University, Chengdu 610065, China}
\email{chenbohui@scu.edu.cn}

\author{Cheng-Yong Du}
\address{School of Mathematical Science, Sichuan Normal University, Chengdu 610066, China}
\email{cyd9966@hotmail.com}

\author{Fengyu Jiang}
\address{School of Mathematics, Sichuan University, Chengdu 610065, China}
\email{jfyscu@foxmail.com}
\thanks{Corresponding author: Fengyu Jiang, Email: \href{mailto:jfyscu@foxmail.com}{jfyscu@foxmail.com}.}

\date{}

\begin{abstract}
In this paper, the 2-group $\SBAut(\CX)$ of  automorphisms of a Lie groupoid $\CX$ is constructed. Considering the 2-group $\SG$ action on $\CX$, we explain the equivalence between 2-group homomorphisms from $\SG$ to $\SBAut(\CX)$ with 
Kan fibrations over $\SG$ with fiber $\CX$. This justifies the notion of Kan fibration for 2-group actions on Lie groupoids. As an application,  we formulate Hamiltonian actions of \'etale Lie 2-groups on orbifolds in terms of Kan fibrations and study the symplectic reductions. We show that, in general, the reduction is in fact a symplectic Lie 2-groupoid, and  under certain isotropic free condition, the reduction is still an orbifold. Also the slice theorem 
of a group $G$ action on Lie groupoids is proved. 
\end{abstract}

\maketitle

\tableofcontents

\section{Introduction}

A Lie groupoid is a fundamental concept in geometry and mathematical physics, closely related to differential and algebraic stacks.  As a special case, an orbifold is equivalent to a proper, \'etale Lie groupoid.  Lie $n$-groupoids, for $0\leqslant n\leqslant\infty$, generalise manifolds and Lie groupoids: a manifold is a Lie $0$-groupoid and a Lie groupoid is a Lie $1$-groupoid.  Studying symmetries---and hence actions---on Lie $n$-groupoids is a natural and fundamental problem.  At present, Kan fibrations of Lie $n$-groupoids provide the standard device for defining actions.

This paper and its sequel aim to understand the symmetries of Lie $n$-groupoids and, consequently, actions on them.  One of the main goals is to clarify how the point of view of Kan fibrations fits into this framework.

Let us recall the definition of a group action on smooth manifolds.  Let $G$ be a group and $X$ be a smooth manifold.  Write $\Diff(X)$ \footnote{For convenience, the product $gh$ is defined to be the composition of diffeomorphisms $h\circ g$.} for the group of diffeomorphisms of $X$. A $G$-action on $X$ is a group homomorphism
\[%
\rho\co G\rto \Diff(X).
\]%
Given such a $\rho$, one may associate to it a smooth map \footnote{The set $G$ is treated as a 0-dimensional manifold with discrete topology.}
\begin{equation}\label{e-Phi}
\Phi\co X\times G\rto X
\end{equation}
such that $\Phi(\cdot,g)$
is $\rho(g)$. Hence $\Phi$ satisfies
\begin{equation}\label{e-Phi-property}
\Phi(x,e)=x    \and   \Phi(\Phi(x,g),h)=\Phi(x,gh).
\end{equation}
The $\Phi$ is the geometric realization of the action $\rho$. Conversely, such a map $\Phi$ determines  a group homomorphism $\rho\co G\to\Diff(X)$. Let $\mathbf H_G(X)$ denote the set of homomorphisms from $G$ to $\Diff(X)$ and
$\mathbf K_G(X)$ denote the set of maps $\Phi$ as above satisfying \eqref{e-Phi-property}. Then $\mathbf H_G(X)$ and $\mathbf K_G(X)$
are in bijective correspondence.
If $G$ is a Lie group.
One says $\rho$ is a smooth action if the corresponding map $\Phi$ in \eqref{e-Phi} is a smooth map.

We  follow this line to understand  actions on Lie $n$-groupoids. In this paper we focus on the case $n=1$, i.e., actions on Lie ($1$-)groupoids. Let $\mc X$ be a Lie groupoid. The study of actions on $\mc X$ consists of the following steps. %

{\bf Step 1: Automorphisms.}
The first issue is to understand the automorphisms of $\CX$. Let $\mathbf{Aut}(\CX)$ denote the group of automorphisms of $\CX$. In \cite{Chen-Du-WangR2019-grp-mor}, it was shown that$\mathbf{Aut}(\CX)$ admits a groupoid structure, i.e., it can be associated with a groupoid $\SAut(\CX)$ such that, at the coarse set level, $|\SAut(\CX)|=\mathbf{Aut}(\CX)$ (cf. Proposition \ref{P-aut-X}). The group structure of $|\SAut(\CX)|$ is reflected on a group-like structure carried by $\SAut(\CX)$. On the other hand, the groupoid $\SAut(\CX)$ with such group-like data can be integrated into a 2-group
$\SB\SAut(\CX)$ (cf. Theorem \ref{T-gpd-gp}).

We remark that it is widely expected that the automorphisms of a Lie groupoid (or, equivalently, a differential stack) form a 2-group, but such a 2-group depends on the choices of automorphisms used. For example, the automorphisms considered in \cite{Li-2015-higher-groupoid} are Hilsum--Skandalis (HS) type. In this paper, an automorphism of $\CX$ is a diagram  \[
\CX\xleftarrow{\phi} \SZ\xrightarrow{\psi} \CX
\]
with hypercovers (cf. Definition \ref{D-morphisms}) $\phi$ and $\psi$.

{\bf Step 2: 2-group actions.}
As $\SB\SAut(\CX)$ is a 2-group, it is natural to consider  2-group actions on $\CX$ via  2-group homomorphisms. Let $\SG$ be a 2-group, it is suggested to define a $\SG$-action on $\CX$ via a 2-group homomorphism or a (weak) morphism
\[%
\rho\co \SG\rto \SB\SAut(\CX)\qq \mbox{or}\qq
\rho\co \SG\rightharpoonup \SB\SAut(\CX).
\]%
Let $\CH_\SG(\CX)$ denote the set of (homo)morphisms from $\SG$ to $\SB\SAut(\CX)$. In \S\ref{Sec-K-vs-H}, we show that we can associate each $\rho\in \CH_\SG(\CX)$ with a Kan fibration
\begin{equation}\label{e-Kan}
\Phi\co \CX\rightarrowtail\msf \SK\twoheadrightarrow \SG.
\end{equation}
The main ingredient is to construct a universal Kan fibration $\SK_\uni$ over $\SB\SAut(\CX)$. The key point is that the Kan fibration plays the role of the map \eqref{e-Phi} (cf. Example \ref{E-smoothaction-manifold}).

\begin{remark} A different approach of considering actions on Lie groupoids is given 
in \cite{Li-2015-higher-groupoid}, where the author defined the
(categorified) action of a (categorified) 2-group $\mc G$ on a Lie groupoid $\CX$  via a morphism in the form of \eqref{e-Phi}.  In  \cite{Li-2015-higher-groupoid}, all morphisms considered are {\em restricted} to be of HS type. Then the author showed that such an action is equivalent to a Kan fibration
over $\SB\mc G$ with fiber $\CX$.
\end{remark}

Let $\CK_\SG(\CX)$ denote the set of Kan fibrations of the form \eqref{e-Kan}. We construct two maps (cf. \S\ref{Sec-2gp-aut})
\[
\Lambda\co \CK_\SG(\CX)\rto \CH_\SG(\CX)\and
\Pi\co \CH_\SG(\CX)\rto\CK_\SG(\CX).
\]
We introduce equivalence relations on $\mc H_\SG(\CX)$ and $\CK_\SG(\CX)$ and denote the set of equivalence classes by $\mathbf H_\SG(\CX)$ and $\mathbf K_\SG(\CX)$ respectively (cf. Definition \ref{D-equiv-H} and Definition \ref{D-equiv-K}). The $\Lambda$ and $\Pi$ then induce maps $\mathbf{\Lambda}$ and $\mathbf{\Pi}$ that give a bijection between $\mathbf H_\SG(\CX)$ and $\mathbf K_\SG(\CX)$ (cf. Theorem \ref{T-HK-equivalence}).

Finally, with these understanding, we are ready to define smooth Lie 2-group actions on Lie groupoids. The smoothness of
a 2-group action given by a 2-group homomorphism $\rho$ from $\SG$ to $\SB\SAut(\CX)$ is defined via the associated Kan fibration. This leads to Definition \ref{D-Kan-smoothaction}.

\begin{remark}
We remark that the arguments in this paper work for $\SG$ to be a Lie 2-groupoid.

One may expect the strategy used here to extend to the case where $\SX$ is a Lie $n$-groupoid with $n>1$. In fact, none of the following questions are trivial:
\begin{enumerate}
\item One may expect $\SB\SAut(\SX)$ to be a $(n+1)$-group. This  suggests that an $(n+1)$-group $\SG$-action on $\SX$ is given by a morphism $\rho$ from $\SG$ to $\SB\SAut(\SX)$, yielding the set $\CH_\SG(\SX)$.
\item Let $\CK_\SG(\SX)$ consist of the Kan fibrations over $\SG$ with fiber $\SX$. It seems not so hard to associate a Kan fibration over $\SG$ with a $(n+1)$-group morphism $\rho\co \SG\rto \SB\SAut(\SX)$, which gives a map $\Lambda\co \CK_\SG(\SX)\rto \CH_\SG(\SX)$.
\item  Constructing the map $\Pi\co \CH_\SG(\SX)\rto \CK_\SG(\SX)$ is the most challenging part:  one must
build a ``universal'' Kan fibration over $\SB\SAut(\SX)$.
\end{enumerate}
All these issues will be addressed in future work.
\end{remark}

In \S\ref{Sec-gp-actionon1} and \S\ref{Sec-2gp-orbifold} we consider two special cases of actions on Lie groupoids:
\begin{itemize}
\item[(I)] 1-group $\SB G$-action on Lie groupoids; in particular we prove the slice theorem in this case;
\item[(II)] 2-group $\SG$-action on orbifolds; we develop the $\SG$-equivariant differential geometry on orbifolds; in particular, we give infinitesimal actions of \'etale Lie 2-groups.
\end{itemize}

In \S\ref{Sec-symplectic-reduction}, we consider an application to symplectic reduction on orbifolds. Consider a symplectic manifold $(X,\omega,G,\mu)$ with a Hamiltonian $G$-action. If $0\in\fk g^*$ is a regular value of the moment map $\mu$, the symplectic reduction $X\sslash G$ may be a symplectic orbifold, and if the action on 0-level set is free, the reduction is still a manifold. We may interpret this as that the symplectic reduction on Lie 0-groupoid $X$ with respect to Hamiltonian 1-group $\SB G$ may be a Lie 1-groupoid; moreover, if the $\SB G$-action on 0-level set is free, the reduction is still a Lie 0-groupoid. Hence, symplectic reduction is {\em not} closed in the category of Lie 0-groupoids.

One may ask: to what extent is the symplectic category  closed under symplectic reduction?  For example,
{\em is the symplectic reduction of an orbifold still an orbifold?}
One of the goal of this paper is to give a complete answer to this question. The answer Theorem \ref{T-orbifold-reduction} has three parts:
\begin{enumerate}
\item If the 2-group $\SG$ is a 1-group $\SB G$, the regular symplectic reduction is still an orbifold; that is, the symplectic reduction of orbifolds with respect to Hamiltonian {1-group} actions remains an orbifold. In this sense, the symplectic reduction with respect to 1-groups is closed in the category of orbifolds.
\item The symplectic reduction of an orbifold with respect to a Hamiltonian \'etale 2-group action may be a Lie 2-groupoid; hence, symplectic reduction is not closed in the category of orbifolds.
\item If the 2-group $\SG$ action on the level set is 1-free (cf. Definition \ref{D-1-free-action}), the symplectic reduction is still an orbifold.
\end{enumerate}

We remark that most of arguments in this paper work for \'etale $\mc X$. A certain special case of the symplectic reduction for symplectic \'etale Lie groupoids was already studied in \cite{Hoffman-Sjamaar2021-IMRN}. Readers are refered to Remark \ref{R-comparison} for the comparison of our work with the work in \cite{Hoffman-Sjamaar2021-IMRN}.

\subsection*{\bf Organization of this paper} We first recall basic concepts of higher Lie groupoids with a focus on the Lie
2-groupoids in \S \ref{Sec-LG}. Then in \S \ref{Sec-aut-X} we study the structure on automorphisms of a Lie groupoid $\CX$ and construct the 2-group $\SB\Aut(\CX)$, and discuss some issues of 2-groups in \S \ref{Sec-unit} and \S \ref{Sec-transformation}. With these preparation, we define and study 2-group actions on Lie groupoids in \S \ref{Sec-2p-action}. Finally, as an application, we study Hamiltonian actions of \'etale Lie 2-groups on symplectic orbifolds and symplectic reductions in \S \ref{Sec-symplectic-reduction}.

\subsection*{Acknowledgements}
This work was supported by the National Natural Science Foundation of China (No. 12071322), the National Key R \& D Program of China (No. 2020YFA0714000).

\section{Lie $n$-groupoids}\label{Sec-LG}

In this section, we review  basic concepts of higher Lie groupoids with a focus on the Lie 2-groupoids. A Lie $n$-groupoid is a simplicial manifold satisfying certain Kan conditions. One may refer to \cite{Zhu2009-IMRN-Lie-n-grpds}, \cite{Gelfand-Manin2003} for further details. See also \cite{Gelfand-Manin2003,Goerss-Jardine2009, Trentinaglia-Zhu2011-Compositio}.

\subsection{Simplicial manifolds}\label{Sec-SM}

For any integer $n\geqslant 0$, denote the set $\{0,1,\ldots, n\}$ by $[n]$.

Throughout this paper, all maps $\zeta\co[m]\rto [n]$ are non-decreasing. We denote an injective map by $\zeta\co [m]\rightarrowtail [n]$
and a surjective map by $\zeta\co [m]\twoheadrightarrow [n]$.

Let $\mc C(m,n)$ denote the set of all maps from $[m]$ to $[n]$, and $\mc C_\infty$ be the union of $\mc C(m,n)$ for all $m$ and $n$.
Let $\mc C^{\triangleleft}(m,n)$ and $\mc C^{\triangleright} (m,n)$ be subsets of injective and surjective maps in $\mc C(m,n)$ respectively and, $\mc C^\triangleleft_\infty$ and $\mc C^\triangleright_\infty$ be corresponding subsets of $\mc C_\infty$.

\begin{defn}
[Simplicial set/manifold]\label{D-simplicial-set}
A {\bf simplicial set} (resp. {\bf simplicial manifold}) $\SX$ is made up of sets (resp. manifolds) $\{X_n\}_{n\geqslant 0}$ and a collection of (resp. smooth) structure maps
\[
\mc S:=\{d_\zeta\co X_n\rto X_m\mid \forall\,\, \zeta\co[m]\rto [n] \mbox{ in } \mc C_\infty\}
\]
satisfying the coherent condition
$d_{\zeta_1\circ \zeta_2}=d_{\zeta_2}\circ d_{\zeta_1}.$
\end{defn}

Set
\begin{equation}\label{e-S-half}
\mc S^\triangleleft:=\{d_\zeta\in \mc S\mid \zeta\in \mc C^\triangleleft_\infty\}  \and
\mc S^\triangleright:=\{d_\zeta\in \mc S\mid \zeta\in \mc C^\triangleright_\infty\}.
\end{equation}

\begin{remark}\label{R-face-maps}
For any $n$ and $0\leqslant i\leqslant n$, define the injection
\[
\zeta^{\triangleleft,i}_n\co [n-1]\rightarrowtail [n]
\]
such that $i\mapsto i+1$ and $\zeta^{\triangleleft,i}_n|_{[i-1]}=id$, and the surjection
\[
\zeta^{\triangleright,i}_n\co [n+1]\twoheadrightarrow [n]
\]
such that $i+1\mapsto i$. Via compositions, $\mc C_\infty$ is generated by $\zeta^{\triangleleft,i}_n, \zeta^{\triangleright,i}_n$ for all $n$ and $0\leqslant i\leqslant n$. Hence $\mc S$ is generated by
\begin{align*}
&d_i^n:=d_{\zeta^{\triangleleft,i}_n}\co X_n \rto X_{n-1} & (\text{face map}),\\
&s_i^n:=d_{\zeta^{\triangleright,i}_n}\co X_{n} \rto X_{n+1} & (\text{degeneracy map}).
\end{align*}
\end{remark}

An element in $X_n$ is called an $n$-simplex. A simplex in the images of degeneracy maps is called degenerate, and otherwise nondegenerate.

\begin{defn}\label{D-skeleton}
The {\bf $n$-th skeleton} $\sk_n(\SX)$ of simplicial set $\SX$ is a simplicial set with
\[
\sk_n(\SX)_m:=X_m,\qq\forall\, m\leqslant n
\]
and $\sk_n(\SX)_m$ only having degenerate simplices in $X_m$ when $m>n$.
\end{defn}

\begin{defn}[homomorphism and $\triangleleft$-homomorphism]\label{D-homomorphism}
Let $\SX$ and $\SY$ be two simplicial sets (resp. manifolds). By a {\bf homomorphism} $\f\co \SX\rto \SY$ we mean a collection of maps (resp. smooth maps) $f_n\co X_n\rto Y_n$, $n\geqslant 0$ that are compatible with structure maps in $\mc S$, i.e,
\[
f_m\circ d_\zeta=d_\zeta \circ f_n, \qq\forall\; \zeta\in \mc C(m,n).
\]
Denote the set of homomorphisms from $\SX$ to $\SY$ by $\Hom(\SX,\SY)$.

If $\f$ is only compatible with structure maps in $\mc S^\triangleleft$, we say $\f$ is a {\bf $\triangleleft$-homomorphism}.
\end{defn}

The fundamental examples of simplicial sets include simplicial sets for the standard simplicial complex $\Delta^m$, its boundary $\partial\Delta^m$ and its $j$-th horn $\Lambda^m_j$ for each $0\leqslant j\leqslant m$. These are described %
in the following examples.

\begin{example}[The simplicial set for $\Delta^m$]
\label{E-Delta}
The {simplicial $m$-simplex $\Delta[m]$} is defined to be
$
(\Delta[m])_n:=\mc C(n,m)
$
or equivalently
\[
(\Delta[m])_n:=\{(i_0,i_1,\ldots,i_n)\mid 0\leqslant i_j\leqslant i_k\leqslant m, \forall\, 0\leqslant j\leqslant k\leqslant n\};
\]
with the structure maps
\begin{align*}
d^n_j(i_0,i_1,\ldots,i_n)&:= (i_0,\ldots,\widehat{i_j},\ldots,i_n),\\
s^n_j(i_0,i_1,\ldots,i_n)&:= (i_0,\ldots,i_{j-1},i_j,i_j,i_{j+1},\ldots,i_n)
\end{align*}
for $j\in \{0,1,\ldots,n\}$, $n\geqslant 0$.
\end{example}

Given a $\zeta\in \mc C(m,n)$, we have the homomorphism  $\msf F_\zeta\co \Delta[m]\to \Delta[n]$ given by
\begin{equation}\label{e-delta-homo}
F_{\zeta,\ell}\co \mc C(\ell,m)\rto \mc C(\ell,n),\qq \zeta'\mapsto \zeta\circ \zeta'.
\end{equation}

\begin{example}[The simplicial set for the boundary $\partial\Delta^m$]\label{E-boundary}
The simplicial set  $\partial\Delta[m]$ of the boundary of $\Delta[m]$ is defined to be $\sk_{m-1}(\Delta[m])$.
\end{example}

\begin{example}[The simplicial set for the horn $\Lambda^m_j$] \label{E-horn}
The simplicial set $\Lambda[m,j]$ of the horn $\Lambda^m_j$ of $\Delta[m]$ is defined to be
\begin{align*}
(\Lambda[m,j])_n&=\{f\in (\Delta[m])_n\mid \{0,\ldots, j-1,j+1,\ldots, m\}\not\subseteq \mathrm{Im}\,f\}\subseteq(\Delta[m])_n.
\end{align*}
\end{example}

It is clear that in the usual sense of simplicial subsets we have embedding homomorphisms
\begin{equation}\label{e-homomorphisms}
\lambda_{m,j}\co\Lambda[m,j]\xrightarrow{\lambda^\circ_{m,j}} \partial\Delta[m] \xrightarrow{\partial_m} \Delta[m].
\end{equation}

Let $\Delta^\bullet_m:=(0,1,\ldots,m)$ be the unique nondegenerated $m$-simplex in $\Delta[m]$. For any $\SX$ and $\f\co \Delta[n]\rto \SX$, $\f$ is uniquely determined by $f_n(\Delta^\bullet_n)$. Hence, \begin{equation}\label{e-n-simplex}
\Hom(\Delta[n],\SX)\cong X_n,\qq
\f\mapsto f_n(\Delta^\bullet_n).
\end{equation}
As a consequence,
\begin{equation*}
d_\zeta=(\SF_{\zeta}) ^\ast\co \Hom(\Delta[n], \SX)\rto \Hom(\Delta[m], \SX)
\end{equation*}
for $\zeta\in \mc C(m,n).$

Set
\begin{eqnarray}
& X_{m,j}=\Hom(\Lambda[m,j],\SX), \qq
& \lambda_{m,j}^\ast\co X_m\rto X_{m,j};
\label{e-a}\\
& X_{\partial m}=\Hom(\partial\Delta[m],\SX),\qq
& \partial_m^\ast\co X_m\rto X_{\partial m}.
\label{e-b}
\end{eqnarray}

\begin{remark}[Expression of $n$-simplices]\label{R-expression}
Let $x\in X_n$ be an $n$-simplex. By \eqref{e-n-simplex}, $x$ can be viewed as a homomorphism $x\co \Delta[n]\rto \SX$ and is identified with $x(\Delta^\bullet_n)$. We usually denote $x(\Delta^\bullet_n)$, and hence $x$, by $x_{0\ldots n}$. The faces of $x_{0\ldots n}$ are denoted by
$x_{i_0\ldots i_k}$ with $0\leqslant i_0<i_1<\ldots<i_k\leqslant n$.
\end{remark}

\subsection{Lie $n$-groupoids}\label{Sec-LnG}

A Lie $n$-groupoid is a simplicial manifold with Kan conditions.

\begin{defn}[Kan condition]\label{D-Kan-condition}
Let $\SX$ be a simplicial set (resp. manifold). We say that $\SX$ satisfies
\begin{itemize}
\item the Kan condition $\Kan(m,j)$ if the $\lambda^\ast_{m,j}$ in \eqref{e-a} is a surjection (resp. a surjective submersion), and
\item the unique Kan condition $\Kan!(m,j)$ if the $\lambda^\ast_{m,j}$ in \eqref{e-a} is a bijection (resp. a diffeomorphism).
\end{itemize}
We say $\SX$ satisfies $\Kan(m)$ (resp. $\Kan!(m)$) if it satisfies $\Kan(m,j)$ (resp. $\Kan!(m,j)$) for all $0\leqslant j\leqslant m$.
\end{defn}

\begin{defn}[Lie $n$-groupoid/group]\label{D-Lie-m}
A {\bf (Lie) $n$-groupoid} is a simplicial set (resp. manifold) that satisfies $\Kan(m)$ for all $1\leqslant m\leqslant n$ and $\Kan!(m)$ for all $m\geqslant n+1$.

A {\bf (Lie) $n$-group} is a (Lie) $n$-groupoid which contains a single $0$-simplex. We usually denote the unique 0-simplex by $\bullet$, and call $e:=s^0_0(\bullet)$ the unit.
\end{defn}

The well-definedness of Lie $n$-groupoids needs extra care, one may refer to \cite[Corollary 2.5]{Henriques2008-Compositio} for details.

By \cite[Proposition 2.37 and Remark 3.22]{Li-2015-higher-groupoid}, an $n$-groupoid is determined by $X_0,\ldots,X_{n+2}$. In fact,
\[
X_m\cong \Hom(\sk_{n+1}\Delta[m],\SX), \qq\forall\, m\geqslant n+3.
\]
In this paper, we are interested in the case that $m\leqslant 2$. Then it suffices to construct $X_m$, $0\leqslant m\leqslant 4$ with structure maps satisfying necessary Kan conditions.

Let $\SX$ and $\SY$ be two Lie $n$-groupoids. One may define an $n$-groupoid $\SHom(\SX,\SY)$ such that $\Hom_0(\SX,\SY)=\Hom(\SX,\SY)$: set \footnote{See for example \cite[Definition 5.1]{Friedman2012-RockyMountain} for the definition of products of simplicial sets.}
\[
\Hom_m(\SX,\SY)=\Hom(\Delta[m]\times \SX,\SY).
\]
For any $\zeta\co [\ell]\rto[m]$, the structure map
\[
d_\zeta\co \Hom_m(\SX,\SY)\rto \Hom_\ell(\SX,\SY)
\]
is given by
\[
(\Delta[m]\times \SX\xrightarrow{\u} \SY)
\mapsto (\Delta[\ell]\times \SX\xrightarrow{(\SF_\zeta,\msf{id})} \Delta[m]\times\SX\xrightarrow{\u} \SY),
\]
where $\SF_\zeta$ is given in \eqref{e-homomorphisms}.

\begin{example}\label{E-X-m}
Let $\SX$ be a Lie $n$-groupoid. For any $k\geqslant 0$, set
$\SX^{[k]}=\SHom(\Delta[k],\SX)$. For any $\zeta\co [\ell]\rto [k]$, we have the homomorphism
\begin{equation}
\d_\zeta=\SF_\zeta^\ast\co \SX^{[k]}\rto \SX^{[\ell]}.
\end{equation}
In particular, we have $\d^m_j$, $\s^m_j$'s.
\end{example}

Properness and \'etaleness are important concepts in Lie groupoids. Here we give the definitions of properness
and \'etaleness for higher Lie groupoids.

\begin{defn}[Proper Lie $n$-groupoid]\label{D proper}
A Lie $n$-groupoid $\SX$ is {\bf proper} if the map $\partial_m^\ast\co X_m\rto X_{\partial m}$ in \eqref{e-b} is proper for all $1\leqslant m\leqslant n$.
\end{defn}

\begin{defn}[\'etale Lie $n$-groupoid and group]\label{R local-Kan}
We say a simplicial manifold $\SX$ satisfies the {\'etale Kan condition} $\Kan!!(m,j)$ if the $\lambda^\ast_{m,j}$ in \eqref{e-a} is an \'etale map, i.e., a local diffeomorphism. We say $\SX$ satisfies $\Kan!!(m)$ if it satisfies $\Kan!!(m,j)$ for all $0\leqslant j\leqslant m$.

A Lie $n$-groupoid $\SX$ is {\bf $m$-\'etale} if it satisfies $\Kan!!(m)$, and {\bf \'etale} if it is $m$-\'etale for all $1\leqslant m\leqslant n$.

A Lie $n$-group $\SX$ is {\bf $m$-\'etale} if it satisfies $\Kan!!(m)$, and {\bf \'etale} if it is $m$-\'etale for all $2\leqslant m\leqslant n$.
\end{defn}

\subsection{Kan fibrations and hypercovers}\label{Sec-HK}
Let $\f\co \SX \rto \SY$ be a morphism of Lie $n$-groupoids. Denote the fibered product of the following diagram
\[
\begin{tikzcd}
&& Y_m=\Hom(\Delta[m],\SY) \arrow[d,"\lambda^\ast_{m,j}"]                  \\
{X_{m,j}=\Hom(\Lambda[m,j],\SX)} \arrow[rr, "f_{m,j}"] && {Y_{m,j}=\Hom(\Lambda[m,j],\SY)}
\end{tikzcd}
\]
by
\[
\Hom(\Lambda[m,j]\rto \Delta[m],\SX \rto \SY).
\]
Then the natural map $\lambda^\ast_{m,j}\co X_m\rto
X_{m,j}$ and the $f_m\co X_m\rto Y_m$ induce a map
\[
\tau_{m,j}=(\lambda^\ast_{m,j},f_m)\co X_m\rto \Hom(\Lambda[m,j]\rto \Delta[m],\SX \rto \SY).
\]
When $m=0$, we take
\[
\Hom(\Lambda[m,j]\rto \Delta[m],\SX \rto \SY):=Y_0
\]
and $\tau_{0,0}:=f_0$. Since $\lambda^\ast_{m,j}\co Y_m\rto Y_{m,j}$ is a surjective submersion,
\[
\Hom(\Lambda[m,j]\rto \Delta[m],\SX \rto \SY)
\]
is a  smooth manifold.

\begin{defn}[Kan fibration]\label{D-Kan-fibration}
A homomorphism $\f\co \SX \rto \SY$ of Lie $n$-groupoids is a {\bf Kan fibration} if $\tau_{m,j}$ are surjective submersions for all $0\leqslant j\leqslant m$, and $\tau_{n,j}$ are diffeomorphisms for all $0\leqslant j\leqslant n$ (which implies that $\tau_{m,j}$ are diffeomorphisms for all $m>n$).

A Kan fibration of Lie $n$-groupoids $\f\co \SX\rto\SY$ is called {\bf $m$-\'etale} if $\tau_{m,j}$ are local diffeomorphisms for all $0\leqslant j\leqslant m$, and {\bf \'etale} if it is $m$-\'etale for all $1\leqslant m\leqslant n-1$.

We denote the conditions that
\begin{itemize}
\item $\tau_{m,j}$ is a surjective submersion by $\Kan(m,j)(\f)$, and,
\item $\tau_{m,j}$ is a local diffeomorphism by $\Kan !!(m,j)(\f)$, and,
\item $\tau_{m,j}$ is a diffeomorphism by $\Kan !(m,j)(\f)$.
\end{itemize}
Similarly, we have $\Kan(m)(\f)$, $\Kan!!(m)(\f)$ and $\Kan!(m)(\f)$.
\end{defn}

A Kan fibration of Lie $n$-groupoids $\f\co \SX\rto \SY$ is treated as a fiber bundle. For each $y_0\in Y_0$, we define the fiber over $y_0$, denoted by $\SFib(y_0)$, to be the fibered product of the following diagram of homomorphisms of simplicial manifolds.
\[
\begin{tikzcd}
& \SX\arrow[d,"\f"]                  \\
{\Delta[0]} \arrow[r, "y_0"] & {\SY}
\end{tikzcd}
\]
The following Lemma is essentially proved in \cite[Lemma 3.36]{Li-2015-higher-groupoid}.

\begin{lemma}\label{L fiber}
Suppose $\f\co \SX\rto \SY$ is a Kan fibration of Lie $n$-groupoids. Then $\SFib(y_0)$ is a Lie $(n-1)$-groupoid.
\end{lemma}

We remark that a Kan fibration of Lie 1-groupoids is the same as a fiber bundle of Lie groupoids.

\begin{example}\label{E-Kan-fibration}
Suppose $\SG$ is a Lie 2-group, i.e., $G_0=\{\bullet\}$. Let $\pi\co\SK\rto\SG$ be a Kan fibration. We say $\SX:=\SFib(\bullet)$ is the fiber of the Kan fibration $\pi$. There is a natural inclusion $\mathsf i\co \SX\rto \SK$, hence we denote the Kan fibration by
\begin{equation}
\pi\co \SX\rightarrowtail \SK \twoheadrightarrow \SG.
\end{equation}
\end{example}

It is natural to define homomorphisms between Kan fibrations. Here we just consider a special case that is need in this paper.

\begin{defn}\label{D-KF-homomorphism}
Let $\pi_i\co \SX_i\rightarrowtail \SK_i\twoheadrightarrow \SG,i=1,2$ be two Kan fibrations over $\SG$. A homomorphism/isomorphism $\SF\co \SK_1\to\SK_2$ is called a {\bf homomorphism/isomorphism of Kan fibrations} if
$\pi_2\circ\SF=\pi_1$. Such an $\SF$ induces a homomorphism/isomorphism on fibers $\f\co \SX_1\to\SX_2$.

If $\SX_1=\SX_2=\SX$, we say $\SF$ is a {\bf $\bullet$-homomorphism/isomorphism of Kan fibrations} if $\f\co \SX\to\SX$ is the identity homomorphism/isomorphism.
\end{defn}

Equivalence is an important concept in the theory of Lie groupoids. For general Lie $n$-groupoids, the notation of hypercover was introduced in \cite[Definition 2.3]{Zhu2009-IMRN-Lie-n-grpds} which plays a role of particular equivalence. The usual equivalence is defined in Definition \ref{D-equivalence}.

For a homomorphism $\msf f\co \SX \rto \SY$ of Lie $n$-groupoids, denote the fibered product of the following diagram
\[
\begin{tikzcd}
&& Y_m=\Hom(\Delta[m],\SY) \arrow[d, "\partial^\ast_m"]                  \\
{ X_{\p m}=\Hom(\partial\Delta[m],\SX)} \arrow[rr, "f_{m,\partial}"] && { Y_{\p m}=\Hom(\partial\Delta[m],\SY)}
\end{tikzcd}
\]
by
\[
\Hom(\partial\Delta[m]\rto \Delta[m],\SX \rto \SY).
\]
Then the natural map $\partial^\ast_m\co X_m\rto  X_{\p m}$ and the $f_m\co X_m\rto Y_m$ induce a map
\[
\tau_m=(\partial^\ast_m,f_m)\co X_m\rto \Hom(\partial\Delta[m]\rto \Delta[m],\SX \rto \SY).
\]
When $m=0$, we take
\[
\Hom(\partial\Delta[m]\rto \Delta[m],\SX \rto \SY):=Y_0
\]
and $\tau_0:=f_0$.

\begin{defn}[Hypercover]\label{D hypercover}
A homomorphism $\f\co \SX \rto \SY$ between Lie $n$-groupoids is a {\bf hypercover} if all
\[
\Hom(\partial\Delta[m]\rto \Delta[m],\SX \rto \SY)
\]
are smooth manifolds \footnote{This fact can be proved inductively for $m\geqslant 1$, see \cite[Lemma 2.4]{Zhu2009-IMRN-Lie-n-grpds}.}, $\tau_m$ are surjective submersions, and $\tau_n$ is a diffeomorphism (which implies $\tau_m$ are diffeomorphisms for all $m>n$).

We denote the condition that $\tau_m$ is a surjective submersion by $\Acyc(m)(\f)$, and the condition that $\tau_m$ is a diffeomorphism by $\Acyc!(m)(\f)$.
\end{defn}

\begin{defn}[Morphism, Morita equivalence and automorphism]\label{D-morphisms}
Given two Lie $n$-groupoids $\SX$ and $\SY$, a {\bf (weak) morphism} from $\SX$ to $\SY$, denoted by $\f\co\SX\rightharpoonup \SY$,
is a diagram
\[
\SX\xleftarrow{\f^\ell} \SZ \xrightarrow{\f^r} \SY
\]
where $\f^\ell$ is a hypercover. We write $\f=(\f^\ell,\SZ,\f^r)$.

If $\f^r$ is also a hypercover, we say $\f$ is a {\bf Morita equivalence} \footnote{The standard definition only requires $\f^\ell$ and $\f^r$ to be equivalence.}. An {\bf automorphism} of $\SX$ is defined to be a Morita equivalence from $\SX$ to itself.
\end{defn}

\begin{example}\label{E-mor-Phi-i}
For any $\zeta\in \mc C^{\triangleleft}(m,n)$, the $\msf d_\zeta\co \SX^{[n]}\rto \SX^{[m]}$ is a hypercover.
\end{example}

\begin{defn}[Equivalence]\label{D-equivalence}
Let $\f\co \SY\rto\SX$ be a homomorphism of Lie $n$-groupoids. We say that $\f$ is an {\bf equivalence} if the homomorphism
\begin{equation}
\d^1_0\circ \msf{pr}_2\co \SY\times_{\f,\SX,\msf d^1_1} \SX^{[1]}\xrightarrow{\msf {pr}_2}\SX^{[1]}\xrightarrow{\msf d^1_0} \SX
\end{equation}
is a hypercover.
\end{defn}

Such an equivalence is also called a weak equivalence in \cite[Definition 3.20]{Behrend-Getzler2017}. Obviously, a hypercover $\f\co\SY\rto\SX$ is an equivalence.

\begin{defn}
[Natural transformation]\label{D-natural-transformation}
Let $\f$ and $\g$ be two homomorphisms from  $\SX$ to $\SY$. A {\bf natural transformation} from $\f$ to $\g$ is a homomorphism $\h\co \SX\rto \SY^{[1]}$ such that
$
\f=\msf d^1_1\circ \msf h$,
$
\g=\msf d^1_0\circ \msf h.
$
We write $\h\co \f\Rightarrow \g$.
\end{defn}

\subsection{Lie 1-groupoid vs Lie groupoid}\label{Sec-LL}

\def \X{\mathsf{X}}

It is well known that (Lie) 1-groupoids are equivalent to (Lie) groupoids. In fact, they are in one-to-one correspondence.

A {\bf groupoid} $\CX$ is a small category such that every morphism has a unique inverse. We denote the set of objects in $\CX$ by $X^0$ and the set of morphisms in $\CX$ by $X^1$. So $X^1$ is the collection of morphisms
\[
X^1=\bigsqcup_{(x,y)\in\, X^0\times X^0} X^1(x,y),
\]
where $X^1(x,y)$ \footnote{In the literature on the theory of categories, this $X^1(x,y)$ is denoted by $\Hom(x,y)$.} is the set of morphisms from $x$ to $y$. We call a morphism $g\in X^1(x,y)$ an arrow from $x$ to $y$, and call $x$ and $y$ the source and the target of $g$ respectively. We write
\[
x=s(g),\qq y=t(g),
\qq
\mbox{and}
\qq
g\co x\rto y\qq (\text{or } x \xrightarrow{g} y),
\]
where $s$, $t\co X^1\rto X^0$ are called the source map
and the target map of the groupoid $\CX$ respectively. Denote the composition of arrows \footnote{In this paper we use the convention that the composition of arrows of a groupoid is going from left to right. %
} by
\[
\cdot\co  X^1(x,y)\times  X^1(y,z)\rto  X^1(x,z), \qq
(g,h)\mapsto g\cdot h\qq (\text{or simply }gh).
\]
For every $x\in X^0$, there exists a unique  unit $1_x\in X^1(x,x)$ with respect to the composition ``$\cdot$'', i.e., $1_x \cdot g=g$ and $g\cdot 1_x=g$. For every arrow $g\in  X^1(x,y)$, there exists a unique inverse $h\in  X^1(y,x)$ such that $g\cdot h=1_x$ and $h\cdot g=1_y$. We denote $h$ by $g^{-1}$. We encode the unit and inverse into two maps:
\begin{itemize}
\item the unit map $u\co X^0\rto  X^1$, $x\mapsto 1_x$;
\item the inverse map $i\co X^1\rto  X^1$, $g\mapsto  g^{-1}$.
\end{itemize}
Therefore a groupoid $\CX$ is a pair of sets $(X^0,X^1)$ with structure maps $(s,t,\cdot,u,i)$.

\begin{defn}[Lie and orbifold groupoid]\label{D Lie-grpd}
For a groupoid $\CX=(X^1\rrto X^0)$, if both $X^0$ and $X^1$ are smooth manifolds and the structure maps are smooth with $s$ and $t$ both being submersions, we call $\CX$ a {\bf Lie groupoid}.

A Lie groupoid $\CX$ is proper if $s\times t\co X^1\rto X^0\times X^0$ is proper and \'etale if both $s$ and $t$ are \'etale, i.e., local diffeomorphisms. 

A proper \'etale Lie groupoid is called an {\bf orbifold (groupoid)}.
\end{defn}

We explain the bijective correspondence
between Lie 1-groupoids and Lie groupoids.

\vskip 0.1in
\noindent
{\em{From Lie 1-groupoids to Lie groupoids.} }
\vskip 0.1in

Let $\X$ be a Lie 1-groupoid. We associate to it  a Lie groupoid $\CX=(X^1\rrto X^0)$:
\begin{itemize}
\item set $X^0:=X_0$, $X^1:=X_1$;
\item set $s:=d^1_1$, $t:=d^1_0$, $u:=s^0_0$;
\item the $\cdot\co X^1\times_{t,X^0,s}X^1\rto X^1$ is the composition of maps
\[
X_{2,1}\xrightarrow{(\lambda^\ast_{2,1})^{-1}}
X_2\xrightarrow{d^2_1}X_1;
\]
\item let $\alpha\in X^1$, the left (resp. right) inverse is defined as the following:  let $\beta_l$ and $\beta_r$ be the 2-simplices of $\X$ such that
\[
\beta_{l,12}=\alpha,\qq \beta_{r,01}=\alpha,\qq
\beta_{l,02}=s^0_0(d^1_0(\alpha)),\and
\beta_{r,02}=s^0_0(d^1_1(\alpha);
\]
the left (resp. right) inverse of $\alpha$ is taken to be $\beta_{l,01}$ (resp. $\beta_{r,12}$);
\item the $\Kan!(3)$ of $\X$ implies that the product ``$\cdot$'' is associative, and the left and right inverses coincide.
\end{itemize}

\vskip 0.1in
\noindent
{\em{From Lie groupoids to Lie 1-groupoids.} }
\vskip 0.1in
Given a Lie groupoid $\CX=(X^1\rrto X^0)$, we associate to it a Lie groupoid $\X$:
\begin{itemize}
\item set $X_0:=X^0$, $X_1:=X^1$;
\item set $d^1_1:=s$, $d^1_0:=t$,
$s^0_0:=u$;
\item the $X_2$ consists of triples
$(\alpha_{01},\alpha_{12},\alpha_{02})\in (X_1)^3$ such that
$\alpha_{01}\alpha_{12}=\alpha_{02}$.
\end{itemize}

\begin{remark}
Under the identification between Lie 1-groupoids and Lie groupoids, one verifies that basic concepts in Lie groupoids such as equivalence, Morita equivalence, automorphism, and etc. coincide with those in Lie 1-groupoids.
\end{remark}

At the end of this section, we give a criterion for a Lie $2$-groupoid to be a Lie $1$-groupoid.

\begin{defn}
[2-isotropy free]
Let $\SZ$ be a Lie 2-groupoid and $z\in Z_0$ be a 0-simplex. Define the 2-isotropic set of $z$ to be
\[
\Gamma_z[2]=\{\alpha\in Z_2\mid
d^2_j(\alpha)=s^0_0(z), \forall\, 0\leqslant j\leqslant 2\}.
\]
The $\SZ$ is called {\bf 2-isotropy free} if, for any $z\in Z_0$, $\Gamma_z[2]$ consists of only one element.
\end{defn}

\begin{lemma}\label{L-property-of-1-free}
Suppose $\SZ$ is a $2$-isotropy free Lie $2$-groupoid. Then
\begin{enumerate}
\item[(1)] for every 2-simplex $\alpha_{012}\in Z_2$ with $\partial^\ast_2(\alpha_{012})=\partial^\ast_2(s^1_j(g))$ for some $g\in Z_1$, we must have $\alpha_{012}=s^1_j(g)$, and
\item[(2)] consequently, if $\partial^\ast_2(\alpha_{012})=\partial^\ast_2(\alpha_{012}')$ we must have $\alpha_{012}=\alpha_{012}'$, i.e., $\partial^\ast_2\co Z_2\rto \p Z_2$ is injective.
\end{enumerate}
\end{lemma}
\begin{proof}
(1). For the first assertion, we prove the case $j=1$, the case $j=0$ can be proved similarly. Suppose $\partial^\ast_2(\alpha_{012})=\partial^\ast_2(s^1_1(g))$. Denote  $z:=d^1_0(g)$. Denote by $1_z:=s^0_0(z)$ and $2_z:=s^1_0(s^0_0(z))=s^1_1(s^0_0(z))$ the degenerate 1-simplex and 2-simplex determined by $z$ respectively. By $\Kan!(3,0)$ of $\SZ$, we have a unique $\alpha_{0123}\in Z_3$ satisfies
\[
\lambda^\ast_{3,0}(\alpha_{0123}) =(\alpha_{012},\alpha_{012},\alpha_{012})\in \Lambda_{3,0}Z.
\]
The $d^3_0(\alpha_{0123})$ has all three faces being $d^2_0(\alpha_{012})=d^2_0(s^1_1(g))=1_z$. Therefore by the $1$-freeness of $\SZ$ we must have $d^3_0(\alpha_{0123})=2_z$. Then by $\Kan!(3,1)$ we must have
\[
\alpha_{0123}= (\lambda^\ast_{3,1})\inv(2_z,\alpha_{012},\alpha_{012}) =s^2_2(\alpha_{012}).
\]
Therefore
\[
\alpha_{012}=d^3_1(\alpha_{0123}) =d^3_1(s^2_2(\alpha_{012}))=s^1_1(g).
\]

(2). We next consider the second assertion. Suppose $\alpha_{012}$ and $\alpha_{012}'$ have the same faces, denoted by $\alpha_{ij}=\alpha_{ij}'$ respectively. Consider the following element
\[
(s^1_1(\alpha_{12})=s^1_1(\alpha_{12}'), \alpha_{012}',\alpha_{012})\in {\Lambda_{3,1}Z},
\]
which by $\Kan!(3,1)$ of $\SZ$ uniquely determines a $\beta\in Z_3$ with $d^3_0(\beta)$ has three faces being
\[
(\alpha_{12}',s^0_0(d^1_0(\alpha_{02})),\alpha_{12})= (\alpha_{12},s^0_0(d^1_0(\alpha_{02})),\alpha_{12})= \partial^\ast_2(s^1_1(\alpha_{12})).
\]
Therefore by the first assertion we have $d^3_0(\beta)=s^1_1(\alpha_{12})$. Also by the $\Kan!(3,2)$ we have
\[
\beta=(\lambda^\ast_{3,2})\inv(s^1_1(\alpha_{12}), s^1_1(\alpha_{12}),\alpha_{012})=s^2_2(\alpha_{012}).
\]
Therefore
\[
\alpha_{012}'=d^3_2(\beta) =d^3_2(s^2_2(\alpha_{012}))=\alpha_{012}.
\]
This finishes the proof.
\end{proof}

This lemma allows us to reduce a proper Lie $2$-groupoid to a proper Lie $1$-groupoid.

\begin{prop}\label{P-2-to-1}
Let $\SZ$ be a $2$-isotropy free, proper Lie $2$-groupoid. Then there is a proper Lie $1$-groupoid $\tilde\SZ$ and a hypercover $\f\co\SZ\rto\tilde\SZ$. If furthermore $\SZ$ is \'etale, then $\tilde\SZ$ is \'etale and every $f_n$ of $\f$ is \'etale.
\end{prop}
\begin{proof}

Denote by $\tilde{\CZ}=(\tilde Z^1\rrto\tilde Z^0)$ the Lie groupoid corresponding to the $\tilde\SZ$ we want. We sketch the construction of $\tilde {\CZ}$.
\begin{enumerate}
\item Set $\tilde Z^0:=Z_0.$
\item Construction of $\tilde Z^1$.
We define an equivalence relation on $Z_1$.
Consider
\[
\mc V=\{z_{012}\in Z_2\mid d^2_0(z_{012})=s^0_0(z_1)\}.
\]
This is a smooth manifold by $\Kan(2,1)$ of $\SZ$.

We say $z\sim z'\in Z_1$ if there exists a $z_{012}\in \mc V$ such that $d^2_2=z$ and $d^2_1=z'$. By $\Kan!(3)$ of $\SZ$, we know this is a well-defined equivalence relation over $Z_1$. Let $\mc Z_1=(\mc V\rrto Z_1)$ be the Lie groupoid describing this equivalence relation.

The $\tilde Z^1$ is defined to be  the coarse space of $\mc Z_1$. Since $\SZ$ is proper, $\mc Z_1$ is proper. By Lemma \ref{L-property-of-1-free}, $\mc Z_1$ is isotropy free. Hence
the coarse space of $\mc Z_1$ is a smooth manifold. We conclude that $\tilde Z^1$ is a smooth manifold;
\item source map: $s([z_{01}]):=d^1_1(z_{01})=z_0$;
\item target map: $t([z_{01}]):=d^1_0(z_{01})=z_1$;
\item unit: $u(z):=[s^0_0(z)]$;
\item multiplication:
let $z_0\xrightarrow{z_{01}}z_1$ and
$z_1\xrightarrow{z_{12}}z_2$, then
\[
[z_{01}]\cdot[z_{12}]:=[d^2_1(z_{012})]
\]
for any $z_{012}=(\lambda_{2,1})^{-1}(z_{01},z_{12})$.   One can verify this is well-defined.
\end{enumerate}

Let $\tilde{\SZ}$ be the Lie 1-groupoid corresponding to $\tilde{\mc Z}$. Consider the natural projection homomorphism $\f\co \SZ\rto\tilde\SZ$ of Lie $2$-groupoids:
\begin{enumerate}
\item $f_0\co Z_0\rto Z_0$ is identity;
\item $f_1\co Z_1\rto \tilde Z_1=Z_1/\sim$ is the projection, i.e., $z_{01}\mapsto [z_{01}]$;
\item $f_2\co Z_2\rto \tilde Z_2$ is the map
\[
z_{012}\mapsto ([z_{01}],[z_{12}],[z_{02}]).
\]
\end{enumerate}
The property $\Acyc!(2)(\f)$ is guaranteed by Lemma \ref{L-property-of-1-free}. Hence $\f$ is a hypercover.

Finally, we prove the \'etaleness of $\tilde\SZ$ and $\f$ under the assumption that $\SZ$ is \'etale.
For $\tilde\SZ$, we only have to prove it is $1$-\'etale. This follows from the $1$-\'etaleness of $\SZ$ as follows. The $1$-\'etalenss of $\SZ$ implies that for each $z_{01}\in Z_1$, there is a neighborhood $U_{z_{01}}\subseteq Z_1$ such that all other 1-simplices in this neighborhood have faces different from those of $z_{01}$, that is $d^1_j$ are both diffeomorphism over $U_{z_{01}}$. So $U_{z_{01}}$ forms a coordinates chart of $\tilde Z_1$. This shows that $\tilde \SZ$ is $1$-\'etale and $f_1$ is \'etale. The \'etaleness of $f_2$ follows from $\Kan!(3)$ and $2$-\'etaleness of $\SZ$.
\end{proof}

\section{The 2-group of automorphisms of
a Lie groupoid}\label{Sec-2gp-aut}

In this section, we study the structure on symmetries of a Lie groupoid $\CX$. We denote by $\SX$ the corresponding Lie 1-groupoid of $\CX$. Let $\mathbf{Aut}(\CX)$ be the set of automorphisms of $\CX$ (cf. Proposition \ref{P-aut-X}). It is known that this set admits groupoid structures. Different choices of automorphism definitions give rise to different groupoids.

In \cite{Chen-Du-WangR2019-grp-mor}, the automorphism is taken to be the one in Definition \ref{D-morphisms}, and the groupoid $\SAut(\CX)$ is constructed.
As usual, the composition of morphisms makes  $\mathbf{Aut}(\CX)=|\SAut(\CX)|$ into a group. It is natural to expect such a group structure to be reflected on $\SAut(\CX)$. We show that $\SAut(\CX)$ carries a so-called semi-strict group-like structure (cf. Definition \ref{D-semi-group-like} and Theorem \ref{T-semi-gl}). Furthermore, we show that such a groupoid can be made into a 2-group (cf. Theorem \ref{T-gpd-gp}). Hence the 2-group $\SB\SAut(\CX)$ is constructed. The details are covered in \S\ref{Sec-aut-X}.

In \S\ref{Sec-unit} and \S\ref{Sec-transformation}, we discuss some issues of 2-groups, such as the units and natural transformations. They play an important role in the study carried out in \S\ref{Sec-2p-action}.

\subsection{The 2-group $\SB\SAut(\CX)$}\label{Sec-aut-X}
We review the groupoid $\SAut(\CX)$ of automorphisms of a Lie groupoid $\CX$ constructed in \cite{Chen-Du-WangR2019-grp-mor}.
\vskip 0.1in \noindent
{\em Definition of $\Aut^0(\CX)$.} It consists of automorphisms of $\CX$ (cf. Definition \ref{D-morphisms}).
\vskip 0.1in
\noindent
{\em Definition of $\Aut^1(\mc X)$.} Let
$\f=(\f^\ell,\SY, \f^r)$, $\g=(\g^\ell,\SZ, \g^r)$.   An arrow from $\f$ to $\g$ is a natural transformation
\[
\alpha\co\f^r\circ \msf{pr}_1\Rightarrow \g^r\circ \msf{pr}_2\co \SY\times_{\f^\ell,\CX,\g^\ell} \SZ\rto \CX.
\]

The following result was proved in \cite{Chen-Du-WangR2019-grp-mor}.

\begin{prop}[\cite{Chen-Du-WangR2019-grp-mor}]\label{P-aut-X}
The $\SAut(\CX)$ is a groupoid. Let the set $\mathbf{Aut}(\CX)$  of automorphisms of $\CX$ to be the coarse set of $\SAut(\CX)$.
\end{prop}

Moreover, the $\SAut(\mc X)$ carries a natural product structure (\cite[\S 4.2]{Chen-Du-WangR2019-grp-mor}), which gives a group-like structure on $\SAut(\CX)$.
For simplicity, we denote $\SAut(\CX)$ by $\mc G$.

\vskip 0.1in
\noindent {\bf (1). Product on $\mc G$.}
The product structure is a homomorphism
$
\star\co \mc G\times \mc G\rto\mc G.
$
For $\f=(\f^\ell,\mathsf{Y},\f^r), \g=(\g^\ell,\mathsf{Z} ,\g^r)$,
\[
\f\star^0\g=
(\f^\ell\circ \msf{pr}_1,\SY \times_{\f^r,\CX,\g^\ell}\SZ,\g^r \circ \msf {pr}_2).
\]
The $\star^1$ is more involved and is given in \cite[\S 4.2]{Chen-Du-WangR2019-grp-mor}. The following result was proved in \cite[Lemma 4.10]{Chen-Du-WangR2019-grp-mor}.
\begin{prop}%
The product $\star$ is associative, modulo the canonical identification
$$
(\SZ_1\times_{\f^r,\CX,\g^\ell}\SZ_2)\times_{\g^r\circ pr_2,\mc X,\msf h^\ell}\SZ_3\cong \SZ_1\times_{\f^\ell,\CX,\g^\ell\circ pr_1}(\SZ_2\times_{\g^r,\mc X,\msf h^r}\SZ_3)$$
of  fibered products.
\end{prop}

\noindent
{\bf (2). Unit of $\mc G$.}
Let $\msf{e}_{\mc X}$, or simply $\msf{e}$, denote the identity automorphism of $\CX$, i.e., $\sf{e}=(id,\CX,id)$. Denote the groupoid $(\{1_{\msf{e}}\}\rrto \{\msf{e}\})$ by $\varepsilon$. The $\varepsilon$ is the unit with respect to the product $\star$ in the sense that both the following two homomorphisms
\begin{align}\label{e-unit}
&\mc G \xrightarrow{({\sf id, p})}\mc G\times \mc G \xrightarrow{\star} \mc G \qq \mathrm{and}\qq
\mc G \xrightarrow{({\sf p,id })}\mc G\times \mc G \xrightarrow{\star} \mc G
\end{align}
are identities, modulo the canonical identification
\[
\SZ\times_{\psi,\CX,id}\CX\cong\SZ\cong \CX\times_{id,\CX,\psi}\SZ
\]
of fibered products, where $\msf p\co \mc G\rto \varepsilon$ denote the trivial homomorphism.

\vskip 0.1in
\noindent
{\bf (3). Inverse of $\mc G$.}
The inverse of $\mc G$ is a homomorphism $i\co\mc G\rto\mc G$ defined as follows.

Definition of $i^0$.
Suppose $\f=(\f^\ell,\SZ, \f^r)\in \mc G^0$,  set
$\iota^0(\f):=(\f^r,\SZ, \f^\ell)$.

Definition of $i^1$.
Let $\alpha$ be an arrow from $\f=(\f^\ell,\SY,\f^r)$ to $\g=(\g^\ell,\SZ, \g^r)$, i.e., $\alpha$ is a homomorphism
\[
\alpha\co  \SY\times_{\f^\ell,\CX,\g^\ell}\SZ\rto \CX^{[1]}.
\]
The $i^1(\alpha)$ is expected to be an arrow from $(\f^r,\SY,\f^\ell)$ to $(\g^r,\SZ,\g^\ell)$, i.e.,  a homomorphism
\[
i^1(\alpha)\co
\SY\times_{\f^r, \CX, \g^r} \SZ \rto \CX^{[1]}.
\]
We sketch the construction of $i^1(\alpha)$ from $\alpha$. As
homomorphisms from
$
\SY\times_{\f^\ell,\mc X,\g^\ell} \SZ\times_{\f^r,\mc X,\g^r} \SY
$
to $\SX$, we have
\[
(\msf d^1_1,\msf d^1_0)\circ \alpha\circ \msf{pr}_{12}=(\f^r\circ \msf{pr}_1, \f^r\circ \msf{pr}_3).
\]
Here $\msf{pr}_{12}$ is the projection to the fibered product of first two terms $\SY\times_{\f^\ell,\mc X,\g^\ell} \SZ$.

\begin{lemma}
There exists a unique homomorphism
\[
\u\co \SY\times_{\f^\ell,\mc X,\g^\ell} \SZ\times_{\f^r,\mc X,\g^r} \SY\rto \SY^{[1]}
\]
such that
\[
\f^{r,[1]}\circ \u= \alpha\circ \msf{pr}_{12}.
\]
\end{lemma}

\begin{proof}
Since $\f^r\co \Y\rto\X$ is a hypercover, $\Y^{[1]}$ is the fibered product of
\[
\X^{[1]}\xrightarrow{(\d^1_1,\d^1_0)} \X\times \X  \and
\Y\times\Y\xrightarrow{(\f^r,\f^r)} \X\times\X.
\]
Furthermore, its projection to $\X^{[1]}$ is same as $\f^{r,[1]}$.

The homomorphisms
\[
\SY\times_{\f^\ell,\mc X,\g^\ell} \SZ\times_{\f^r,\mc X,\g^r} \SY\xrightarrow{\msf{pr}_{12}}
\SY\times_{\f^\ell,\mc X,\g^\ell} \SZ\xrightarrow{\alpha} \X^{[1]}
\]
and
\[
\SY\times_{\f^\ell,\mc X,\g^\ell} \SZ\times_{\f^r,\mc X,\g^r} \SY\xrightarrow{(\msf{pr}_{1},\msf{pr}_3)}
\SY\times\SY
\]
imply the homomorphism $\u$.
\end{proof}

Then $\u$ determines $i^1(\alpha)$ uniquely such that
\[
i^1(\alpha)\circ \msf{pr}_{23}
=\f^{\ell,[1]}\circ \msf u.
\]
Moreover, $i=(i^0,i^1)\co\mc G\rto \mc G$ is a homomorphism of inverse.

We next study the properties of the homomorphism of inverse. Let $\f\co\SY\rto \SX$ be a homomorphism of Lie 1-groupoids. Define the kernel of $\f$ by
\[
\ker(\f):=\SY^{[1]}\times _{\f^{[1]}, \SX^{[1]}, \mathsf s^0_0} \SX.
\]
The projection $\msf{pr}_1\co \ker(\f)\rto \SY^{[1]}$ defines an embedding of $\ker(\f)$ into $\SY^{[1]}$. So we may treat $\ker(\f)$ as a subgroupoid of $\SY^{[1]}$.

\begin{lemma}
Suppose $\f\co \SY\rto\SX$ is a hypercover, then
\begin{equation}
\ker({\f})\cong \SY\times_{\f,\SX,\f}\SY.
\end{equation}
\end{lemma}

\begin{prop}
The homomorphisms
\begin{align*}
&\CG\xrightarrow{({id,\iota})}\CG\times \CG \xrightarrow{\star} \CG \and \CG\xrightarrow{({\iota,id})} \CG\times \CG   \xrightarrow{\star} \CG
\end{align*}
are both natural transformed to $\msf p\co\CG\rto\varepsilon\subseteq\CG$. Suppose the natural transformations are denoted by
\begin{equation}\label{e-inverse}
\psi_\ell\co (\star\circ({\iota, id}))\Rightarrow\msf p \and \psi_r\co (\star\circ({id, \iota}))\Rightarrow\msf p.
\end{equation}
\begin{itemize}
\item Description of $\psi_\ell(\f)$ for an $\f\in \mc G^0$. The $\psi_\ell(\f)$ is a natural transformation
from $i^0(\f)\star \f$  to $\msf e$, which is given by
\begin{equation}
\psi_\ell(\f)=\left(\f^r\co \ker({\f^\ell})\rto \mc X^{[1]}\right).
\end{equation}
\item Description of $\psi_r(\f)$ for an $\f\in \mc G^0$. The $\psi_r(\f)$ is a natural transformation from $\f\star i^0(\f)$ to $\msf e$, which is given by
\begin{equation}
\psi_r(\f)=\left(\f^\ell\co \ker({\f^r})\rto \mc X^{[1]}\right).
\end{equation}
\end{itemize}
\end{prop}

The proof is direct, we omit it.

\begin{defn}\label{D-semi-group-like}
We say a groupoid $\mc G$ is {\bf semi-strict group-like} if it is equipped with the following structures:
\begin{enumerate}
\item an associative product $\star\co \CG\times\CG\rto \mc G$,
\item a unit element $\varepsilon=(\{1_e\}\rrto \{e\})$ for $e\in G^0$ such that \eqref{e-unit} holds,
\item an inverse homomorphism $i\co \CG\rto \CG$ with natural transformations \eqref{e-inverse}.
\end{enumerate}
\end{defn}

So far, we have shown  that

\begin{theorem}\label{T-semi-gl}
The $\SAut(\CX)$ is a semi-strict group-like groupoid.
\end{theorem}

Such a semi-strict group-like structure over $\SAut(\CX)$ can be integrated into a 2-group.

\begin{theorem}\label{T-gpd-gp}
A semi-strict group-like groupoid $\mc G=(\mc G^1\rrto \mc G^0)$ can be canonically associated with a 2-group
$\SB \mc G$. As a consequence, we have a 2-group $\SB\SAut(\mc X)$.
\end{theorem}
\begin{remark}
There are other versions for endowing groupoid with group-like structures. The most strict one is a groupoid that generated from  a crossed module of groups;  the most  general one is given in \cite{Hoffman-Sjamaar2021-IMRN}, with all properties  (1)-(3) hold only up to certain natural transformations. The structure we adopt here lies in between, so we call it a semi-strict structure.
In \cite{Li-2015-higher-groupoid}, $\SB\CG$ is already constructed. Here, for the sake of self-contained,   we outline the construction.
\end{remark}

\begin{proof}[Proof of Theorem \ref{T-gpd-gp}] We construct $\sk_2,\sk_3,\sk_4$ of $\sf B\mc G$ step by step. Here, we assume $\sk_n$ is available once $(BG)_k,k\leqslant n$ are constructed.

\noindent {\bf Step 1. description of $\sk_2(\sf B\mc G)$}.
Set
\[
(BG)_0:=\{\bullet\} \and (BG)_1:=\mc G^0,
\]
and
\[
(BG)_2:=\left\{\alpha_{012}=(\alpha_{12}, \alpha_{02},\alpha_{01}, g_{012})\,\left|\,
\alpha_{01},\alpha_{02},\alpha_{12}\in \CG^0,
g_{012}\in \CG^1(\alpha_{01}\star \alpha_{12},\alpha_{02})\right.\right\}.
\]
The face maps $d^2_j$ are $d^2_0(\alpha_{012})=\alpha_{12}, d^2_1(\alpha_{012})=\alpha_{02}$ and $d^2_2(\alpha_{012})=\alpha_{01}$. The degeneracy maps are $s^0_0(\bullet)=e$ and
\[
s^1_0(\alpha_{01})=(\alpha_{01},\alpha_{01},e,1_{e\star \alpha_{01}}=1_{\alpha_{01}}) \and s^1_1(\alpha_{01})=(e,\alpha_{01},\alpha_{01}, 1_{\alpha_{01}\star e}=1_{\alpha_{01}}).
\]
Obviously, they satisfy the coherent condition (cf. Definition \ref{D-simplicial-set}). %

\noindent
{\bf Step 2. description of $\sk_3(\sf B\mc G)$}.
$(BG)_3$ is a subset of
\[
\widetilde{BG}_3:=\Hom(\partial\Delta[3], \sk_2(\SB\mc G)) \subseteq ((BG)_2)^4.
\]
We explain the condition for an element $\zeta\in \widetilde{BG}_3$ to be in $(BG)_3$. Suppose
\[
\alpha_{0123}=(\alpha_{123},\alpha_{023}, \alpha_{013},\alpha_{012})\in \widetilde{BG}_3\subseteq ((BG)_2)^4
\]
where $\alpha_{ijk}=(\alpha_{jk},\alpha_{ik}, \alpha_{ij},g_{ijk})$. Then we say $\alpha_{0123}\in (BG)_3$ if and only if
\begin{align}\label{E def-BG3}
(1_{\alpha_{01}}\star g_{123})\cdot g_{013} =(g_{012}\star 1_{\alpha_{23}})\cdot g_{023}.
\end{align}
The face maps $d^3_j(\alpha_{0123})$ are defined in an obvious way. The degeneracy maps $s^2_j\co (BG)_2\rto (BG)_3$ are given by
\begin{align*}
s^2_0(\alpha_{012})&=(\alpha_{012},\alpha_{012}, s^1_0(\alpha_{02}), s^1_0(\alpha_{01})),\\
s^2_1(\alpha_{012})&=(s^1_0(\alpha_{12}),\alpha_{012}, \alpha_{012},s^1_1(\alpha_{01})),\\
s^2_2(\alpha_{012})&=(s^1_1(\alpha_{12}),s^1_1(\alpha_{02}), \alpha_{012},\alpha_{012}).
\end{align*}
They also satisfy the coherent conditions.

\noindent {\bf Step 3. description of $\sk_4(\sf B\mc G)$}. Set
\[
(BG)_4:=\Hom(\partial\Delta[4], \sk_3(\sf B\mc G)).
\]
The face maps and degeneracy maps are defined in
a natural way.

It is easy to check that these structure maps between $(BG)_n$ for $0\leqslant n\leqslant 4$ satisfies the coherent condition.  For Kan conditions, it is easy to check the  $\Kan(n)$ for $n\leqslant 2$, but the verification of $\Kan!(3)$ and $\Kan!(4)$ is lengthy, albeit straightforward. We omit all computations.
\end{proof}

\subsection{Units of a 2-group}
\label{Sec-unit}

Let $\SG$ be a 2-group with unit $e$. Recall that $\SG$ consists of a collection of sets $\{G_n,n\geqslant 0\}$ and structure maps $\mc S$.

\begin{defn}\label{D-a-transformation}
Let $\alpha \in G_1$.
We say $\Omega_{0\cdots n}\in \SG^{[1]}_n$ is an {\bf $\alpha$-cylinder} if $\Omega_i=\alpha$ for every $0\leqslant i\leqslant n$. Let $G^{[1]}_{n,\alpha}$ denote the set of
$\alpha$-cylinders.
\end{defn}

\begin{defn}\label{D-transformation}
A {\bf transformation} of $\SG$ is an $e$-cylinder in $G^{[1]}_{1,e}$. Let $\Omega\in G^{[1]}_{1,e}$; we say it is a transformation from $\d^1_1(\Omega)$ to $\d^1_0(\Omega)$.

We say $e'\in G_1$ is a {\bf unit element} if there exists a transformation $\omega$ from $e$ to $e'$. We may denote $e'$ by $e_\omega$.
\end{defn}

The following proposition says that any unit element of $\SG$ can serve as a unit of $\SG$ after we modify its degeneracy maps.

\begin{prop}
Let $e_\omega$ be a unit element of $\SG$ via the transformation $\omega$. Then $\{G_n\}$ admits a new 2-group structure ${\mc S}_\omega$ such that $e_\omega$ is the unit and ${\mc S}^\triangleleft_\omega=\mc S^\triangleleft$. We denote this new 2-group by $(\SG,e_\omega)$.
\end{prop}
\begin{proof}
We define new degeneracy maps $s^n_{\omega,j}$. For $n=0$, set $s^0_{\omega,0}(\bullet)=e_\omega$.

Now consider $s^1_{\omega,i}$. For simplicity, suppose $i=1$. Let $g_{01}\in G_1$ and $h_{012}=s^1_1(g_{01})$. Then
\[
h_{01}=h_{02}=g_{01},\and h_{12}=e.
\]
We construct an element $\Omega_{012}\in G^{[1]}_2$ such that
\[
\d^1_1(\Omega_{012})=h_{012},\qq
\Omega_{01}=\Omega_{02}=\s^0_0({g_{01}}),\and
\Omega_{12}=\omega.
\]
Here $\s^n_j$ is as in Example \ref{E-X-m}. By $\Kan!(3)$ of $\SG$, there are exactly one such $\Omega_{012}$. Then we set
\[
s^1_{\omega,1}(g_{01}):=\d^1_0(\Omega_{012}).
\]

Similar constructions are applied to $s^n_{\omega,i}$, $n\geqslant 1$. Let $g_{0\ldots n}\in G_n$ and $h_{0\ldots (n+1)}=s^n_i(g_{01\ldots n})$. Then $h_{i(i+1)}=e$ and $h_{jk}=g_{\zeta(j)\zeta(k)}$ otherwise, where $\zeta=\zeta^{\triangleright,i}_n$.
There exists a unique  $\Omega_{0\ldots (n+1)}\in G^{[1]}_{n+1}$ such that
\begin{equation*}
\msf d^1_1(\Omega_{0\ldots (n+1)})=h_{0\ldots(n+1)},\qq
\Omega_{i(i+1)}=\omega,\and
\Omega_{jk}=\s^0_0({g_{\zeta(j)\zeta(k)}}).
\end{equation*}
We then set
\[
s^n_{\omega,i}(g_{0\ldots n}):=\msf d^1_0 (\Omega_{0\ldots (n+1)}).
\]
By the next Lemma \ref{L-tech}, one can show that
$\mc S_\omega=\{d^n_i, s^n_{\omega,j}\}$ gives a $2$-group structure on $\{G_n\}$. We omit the details.
\end{proof}

The following lemma is a useful technique result, the proof is straightforward, so we only state it.

\begin{lemma} \label{L-tech}
Let $\Omega$, $\Omega'$ and $\Omega''$ be elements in $G^{[1]}_n$, $n\geqslant 2$ such that
\[
\d^1_1(\Omega)=\d^1_1(\Omega'') \and
\d^1_0(\Omega)=\d^1_1(\Omega'),
\]
and furthermore, for any pair $0\leqslant i<j\leqslant n$,
\begin{enumerate}
\item $\Omega_{ij}$, $\Omega'_{ij}$ and $\Omega_{ij}''$ are $e$-cylinders in $G^{[1]}_{1,e}$;
\item there exists $\Theta_{ij}\in G^{[2]}_1$ such that
\[
\d^2_1(\Theta_{ij})=\Omega_{ij},\qq
\d^2_0(\Theta_{ij})=\Omega'_{ij},\and
\d^2_2(\Theta_{ij})=\Omega''_{ij};
\]
\item $\Theta_{ij}$ is either $\msf s^1_0(\Omega_{ij})$ or $\msf s^1_1(\Omega'_{ij})$.
\end{enumerate}
Then
$
\d^1_0(\Omega')=\d^1_0(\Omega'').
$
\end{lemma}

Now we consider the 2-group $\SB\SAut(\CX)$. The unit $\msf e$ is the identity automorphism. We are in particular interested in a unit element $\bfe$, the automorphism of $\CX$ given by
\[
\CX\xleftarrow{\msf s} \CX^{[1]}\xrightarrow{\msf t}\CX,
\qq \mbox{ or equivalently }\qq
\SX\xleftarrow{\d^1_1}\SX^{[1]} \xrightarrow{\d^1_0}\SX.
\]
There is a canonical natural transformation
$$
\omega: \msf e\Rightarrow \bfe
$$
which is the identity homomorphism
\[
\omega: \SX\times_\SX \SX^{[1]}\cong \SX^{[1]}\to \SX^{[1]}.
\]
Hence, $\omega$ corresponds to a 2-simplex
$\omega_{012}$ in $\SB\SAut(\CX)$ with $\omega_{01}=\omega_{12}=\msf e$ and $\omega_{02}=\bfe$.
Set
\begin{equation}
{\omega}_\bullet:=\pi_{1,1}^\ast(\omega_{012}).
\end{equation}
Here $\pi_{1,1}$ is given by \eqref{e-pi-n1} below.

\begin{lemma}
The $\omega_\bullet$ is a transformation from $\msf e$ to $\bfe$.
\end{lemma}

This follows directly from the definition.

\begin{remark}
Consider $\Delta[n]\times \Delta[1]$. We denote its 0-simplices by
\[
00,\ldots, n0; 01,\ldots, n1.
\]
It contains an $n+1$-simplex $\Delta[n+1]$ with vertices
\[
00,\ldots,n0,n1.
\]
Let
\begin{equation}\label{e-pi-n1}
\pi_{n,1}\co \Delta[n]\times \Delta[1]\rto \Delta[n+1]
\end{equation}
be a homomorphism that maps $i0$ to $i0$ and
$i1$ to $n1$.
\end{remark}

Set
\begin{equation}
(\SB\SAut(\CX),\bfe):=(\SB\SAut(\CX), e_{\omega_\bullet}).
\end{equation}

\subsection{Natural transformations of 2-group homomorphisms}\label{Sec-transformation}

We have already defined natural transformations for Lie $n$-groupoids (cf. Definition \ref{D-natural-transformation}). In this subsection we focus on natural transformations for Lie $2$-groups.

\begin{defn}
Let $(\SH,e_\SH)$ and $(\SG,e_\SG)$ be 2-groups.
A ($\triangleleft$-)homomorphism $\SN\co \SH\to \SG^{[1]}$ is called a {\bf natural ($\triangleleft$-)transformation} if $N_0(\bullet)=e_\SG$. We call $\d^1_1(\SN)$ and $\d^1_0(\SN)$ to be the source and target of $\SN$, respectively.
\end{defn}

\begin{lemma}
Let $\SN$ be a natural $\triangleleft$-transformation as above. For any $h\in H_1$, $N_1(h)$
is a transformation of $\SG$.
\end{lemma}
\begin{proof}
Since $\SN$ is compatible with face maps,
\[
d^1_i(N_1(h))=N_0(d^1_i(h))=N_0(\bullet)=e_\SG.
\]
Hence $N_1(h)$ is a $e_\SG$-cylinder, which is a transformation of $\SG$.
\end{proof}

\begin{lemma}
Let $\SN$ be a natural transformation as above. Then $N_1(e_\SH)=\s^0_0(e_\SG)$.
\end{lemma}
\begin{proof}
Since $\SN$ is compatible with the degeneracy maps,
\[
N_1(e_\SH)=N_1(s^0_0(\bullet))=
s^{0}_0(N_0(\bullet))=s^{0}_0(e_\SG).
\]
\end{proof}

\begin{lemma}
Suppose $\SN\co\SH\rto \SG^{[1]}$ is a natural $\triangleleft$-transformation. Then $\SN$ is determined by $\d^1_1\circ\SN$ and $N_1$. Conversely, given a $\triangleleft$-homomorphism $\f\co \SH\rto\SG$ and a map $N_1\co H_1\rto G^{[1]}_{1,e}$ such that
$
f_1=\d^1_1\circ N_1,
$
there exists a unique natural $\triangleleft$-transformation $\SN\co \SH\rto \SG^{[1]}$ such that the source of $\SN$ is $\f$, i.e,
$
\d^1_1\circ\SN=\f.
$
\end{lemma}
\begin{proof}
We show ``$\Leftarrow$''. Suppose $\f$ and $N_1$ are given. We construct $N_2\co H_2\rto G^{[1]}_2$. Let $h\co\Delta[2]\rto \SH$ be an element in $H_2$ and denote $N_2(h)$ by $g\co\Delta[2]\times\Delta[1]\rto \SG$. By $\Kan!(3)$ of $\SG$, there exists a unique $g\in G^{[1]}_2$ such that
\[
\d^1_1(g)=f_2(h) \and  g_{ij}= N_1(h_{ij}).
\]
The same argument applies to $N_n$ for all $n\geqslant 3$. We get a natural $\triangleleft$-transformation $\SN$.
\end{proof}

\begin{prop}
Suppose $\SN\co \SH\rto \SG^{[1]}$ is a natural $\triangleleft$-transformation and $\d^1_1\circ \SN$ is a homomorphism. Let $\omega=N_1(e_\SH)$. Then
\[
\d^1_0\circ \SN\co (\SH,e_\SH)\to (\SG, e_\omega)
\]
is a homomorphism.
\end{prop}
\begin{proof}
Set $\u=\d^1_1\circ \SN$ and $\msf v=\d^1_0\circ \SN$. We show $\msf v$ is compatible with degeneracy maps, i.e,
\begin{equation}\label{e-to-prove}
s^n_{\omega,j}\circ v_n=v_{n+1}\circ s^n_j.
\end{equation}
For simplicity, we assume $j=n$. Let $h_{0\ldots n}\in H_n$, set
$
\tilde h_{0\ldots (n+1)}=s^n_n(h_{0\ldots n}).
$ Let
\[
g_{0\ldots n}=u_n(h_{0\ldots n}),\;\;\;
g'_{0\ldots n}=v_n(h_{0\ldots n}).
\]
Let
\[
\Omega''_{0\ldots (n+1)}=N_{n+1}(\tilde h_{0\ldots (n+1)}).
\]
To show \eqref{e-to-prove}, it is sufficient to show
\[
\d^1_0(\Omega''_{0\ldots (n+1)})
=s^n_{\omega,n}(g'_{0\ldots n}),
\]
since the left hand side of the equation is $v_{n+1}(s^n_n(h_{0\ldots n}))$.

By the assumption that $\u$ commutes with degeneracy maps,
\[
\d^1_1(\Omega''_{0\ldots (n+1)})
=s^n_n(g_{0\ldots n})=:\tilde g_{0\ldots (n+1)}.
\]
Then $\tilde g_{n(n+1)}=e$. We construct $\Omega_{0\ldots(n+1)}$, $\Omega'_{0\ldots (n+1)}$ in $G^{[1]}_{n+1}$:
\begin{itemize}
\item $\Omega_{0\ldots (n+1)}$ satisfies the condition
\[
\d^1_1(\Omega_{0\ldots(n+1)})=\d^1_1(\Omega''_{0\ldots(n+1)}), \qq \Omega_{n(n+1)}=s^0_0(e_G),\and
\Omega_{ij}=N_1(\tilde h_{ij}),
\]
in the last term, $(ij)\not=(n(n+1))$; by the construction, we see that
\begin{equation*}
\d^1_0(\Omega_{0\ldots(n+1)})=s^n_n(g'_{0\ldots n}) =:\tilde g'_{0\ldots (n+1)}.
\end{equation*}
\item $\Omega'_{0\ldots (n+1)}$ satisfies the condition
\[
\d^1_1(\Omega'_{0\ldots(n+1)})
=\d^1_0(\Omega_{0\ldots(n+1)}),\qq
\Omega'_{n(n+1)}=N_1(e_\SH)=\omega,\and
\Omega'_{ij}=s^0_0(\tilde g'_{ij}).
\]
\end{itemize}
One can verify that $\Omega,\Omega'$ and $\Omega''$
satisfy the hypotheses of Lemma \ref{L-tech}, yielding
\[
\d^1_0(\Omega'_{0\ldots(n+1)})=
\d^1_0(\Omega''_{0\ldots (n+1)}).
\]
The right hand side is precisely $s^n_{\omega,j}( g'_{0\ldots n})$. Indeed,  $\d^1_1(\Omega')$ is $s^n_n(g')$, by the definition of $\Omega'$, $\d^1_0(\Omega')$
is $s^n_{\omega,n}(g')$.
\end{proof}

Now we focus on the special case that $\SH=\SG$.

\begin{coro} \label{C-transformation}
Let $(\SG,e)$ be a 2-group. Let $\SN\co \SG\rto \SG^{[1]}$ be a natural $\triangleleft$-transformation from identity to $\u=\d^1_0\circ \SN$. Suppose $\omega=N_1(e)$. Then
\[
\u \co(\SG,e)\rto (\SG,e_\omega)
\]
is a 2-group homomorphism. Furthermore, if $N_1(e_\omega)=s^0_0({e_\omega})$, then $\u\co (\SG,e_\omega)\rto (\SG,e_\omega)$ is also a 2-group homomorphism.
\end{coro}

\begin{coro}
Given $\omega$ be a transformation from $e$ to $e_\omega$, set a natural $\triangleleft$-transformation $\SN$ such that
\[
N_1(g)=\left\{
\begin{array}{ccc}
\omega & \mbox{if} & g=e,\\
s^0_0(g) &\mbox{if} &g\not=e.
\end{array}
\right.
\]
Then
$\Psi_\omega:=\d^1_0(\SN)$ is a 2-group homomorphism from $(\SG,e_\omega)$ to itself.
\end{coro}

\begin{remark}
Let $\rho\co\SH\rto (\SG,e)$ be a 2-group homomorphism. Then $\rho$ is only a $\triangleleft$-homomorphism from $\SH$ to $(\SG,e_\omega)$.
On the other hand, by composing with $\Psi_\omega$, then $\Psi_\omega\circ \rho\co \SH\rto (\SG, e_\omega)$ is a 2-group homomorphism.
\end{remark}

We apply the construction above to the 2-group $(\SB\CG,\msf e)$ where $\CG=\SAut(\CX)$. First, recall that we have a transformation $\omega_\bullet$ from $\msf e$ to $\bfe$. Hence, we have a 2-group $(\SB\CG,\bfe)$ and a homomorphism $\Psi\co(\SB\CG,\msf e)\rto (\SB\CG,\bfe)$.

Now we introduce a natural $\triangleleft$-transformation $\SN$, This is crucial  for constructing a universal Kan fibration $\SK_{\mathrm{uni}}$ over $\SB\SAut(\mc X)$ (cf. \S\ref{Sec-2p-action}).

We first construct a natural transformation
$\mc N\co \CG^0\rto \CG^1$ whose source is
$\msf {id}\co \CG\to\CG$.
\vskip 0.1in
\noindent
{\bf Step 1. Construction of  $\mc L^0:=t\circ\mc N$.}
\vskip 0.1in
 Let
$\f=(\f^\ell,\SZ,\f^r)\in \CG^0$ be an automorphism of $\CX$. We associate to it an automorphism $\g=(\g^\ell, \SY,\g^r)$ which we set  to be $t(\mc N(\f))$.

{\em Step 1.1.  Construction of  $Y^0$.} Consider the space $$
W:=X^1\times_{t,X^0,f^{\ell,0}} Z^0\times_{f^{r,0},X^0,s} X^1
$$ with an element $(a,z_0,b)$ illustrated as
\[
\xymatrix{
x_0\ar[r]^-a & y_0 &z_0 \ar@{-->}[l]_-{f^{\ell,0}} \ar@{-->}[r]^-{f^{r,0}} & w_0\ar[r]^-b & u_0.}
\]
Since both $\f^\ell$ and $\f^r$  are hypercovers, the space $W$ is a smooth manifold. $W$ admits a (right) $\SZ$-action (cf. \cite[Definition 2.14]{Adem-Leida-Ruan2007} and \cite[Section 5.3]{Moerdijk-Mrcun2003}) whose anchor map is the projection $pr_2$, the action map is given by
\begin{align*}
W\times_{{pr_2},X^0,s} Z^1&\rto W\\
((a,z_0,b),c)&\mapsto (a\cdot f^{\ell,1}(c),t(c), [f^{r,1}(c)]\inv\cdot b),
\end{align*}
i.e.,
\[
\xymatrix{
x_0\ar[r]^-a & y_0 \ar[d]_-{f^{\ell,1}(c)}&z_0 \ar[d]_-c \ar@{-->}[l]_-{f^{\ell,0}} \ar@{-->}[r]^-{f^{r,0}} & \ar[d]^-{f^{r,1}(c)} w_0\ar[r]^-b & u_0\\
& y_0' &z_0' \ar@{-->}[l]_-{f^{\ell,0}} \ar@{-->}[r]^-{f^{r,0}} & w_0'.&}
\]
Define $Y^0$ to be the coarse space of the action Lie groupoid $W\rtimes \SZ$.

\begin{lemma}
The $Y^0$ is a smooth manifold.
\end{lemma}
\begin{proof}
It can be directly checked that the groupoid $W\rtimes \SZ$ is proper and isotropy free, hence, is Morita equivalent to a smooth manifold. Consequently, its coarse space, $Y^0$, is a smooth manifold.
\end{proof}

{\em Step 1.2. Construction of $Y^1$.} Set
\[
Y^1:=X^1\times_{s,X^0,s \circ pr_1}  Y^0\times_{t \circ pr_3,X^0,s} X^1,
\]
of which an element $(c,[a,z_0,b],d)$ is illustrated as
\[
\xymatrix{
x_0'& \ar[l]_-c [x_0\ar[r]^-a & y_0 &z_0 \ar@{-->}[l]_-{f^{\ell,0}} \ar@{-->}[r]^-{f^{r,0}} & w_0\ar[r]^-b & u_0]\ar[r]^-d &u_0'.}
\]
Clearly, $Y^1$ is a smooth manifold.

\begin{lemma}\label{L construct-CY}
The structure maps described below makes $\SY:=(Y^1\rrto Y^0)$ into a Lie groupoid:
\begin{align*}
s(c,[a,z_0,b],d)&=[a,z_0,b],\\
t(c,[a,z_0,b],d)&=[c\inv \cdot a,z_0,b \cdot d],\\
(c,[a,z_0,b],d)\cdot(c',[c\inv\cdot a ,z_0,b\cdot d],d')&=(c\cdot c',[a,z_0,b],d\cdot d'),\\
(c,[a,z_0,b],d)\inv& =(c\inv,[c\inv\cdot a,z_0,b\cdot d],d\inv),\\
1_{[a,z_0,b]}&=(1_{s(a)},[a,z_0,b],1_{t(b)}).
\end{align*}
\end{lemma}

{\em Step 1.3. Construction of $\g$.}
The hypercovers $\g^\ell\co\SY\rto \CX$ and $\g^r\co\SY\rto \CX$ are given by
\begin{align*}
g^{\ell,0}([a,z_0,b])=s(a),\qq  & g^{\ell,1}(c,[a,z_0,b],d)=c, \\
g^{r,0}([a,z_0,b])=t(b),\qq & g^{r,1}(c,[a,z_0,b],d)=d.
\end{align*}
Set $\g=(\g^\ell,\SY,\g^r)$.

\vskip 0.1in\noindent
{\bf Step 2. Construction of $\mc N$.}
\vskip 0.1in

By  definition,
$\mc N(\f)$ is a natural transformation from $\f$ to $\g=\mc L^0(\f)$. It is defined by
\[
\mc N(\f)\co \SZ^0\times_{f^{\ell,0},X^0,g^{\ell,0}}Y^0\rto X^1,\qq (z_0,[a,z_0',b])\mapsto f^{r,1}\{(f^{\ell,1})\inv(a)\}\cdot b,
\]
and is illustrated in the following diagram
\[
\xymatrix{
z_0\ar@{-->}[d]^-{f^{\ell,0}} \ar@/^/[drr]^-{(f^{\ell,1})\inv(a)}\ar@{-->}[rrr]^-{f^{r,0}}  &&&f^{r,0}(z_0)\ar[d]^-{f^{r,1}\left[(f^{\ell,1})\inv(a)\right]} \ar@/^1.5pc/[dr]^-{\mc N(\f)(z_0,[a,z_0',b])}&
\\
x_0'\ar[r]^-a &y_0' &z_0'\ar@{-->}[l]_-{f^{\ell,0}} \ar@{-->}[r]^-{f^{r,0}} & w_0'\ar[r]^-b &u_0'.}
\]
This finishes the construction of $\mc N$.

Set
\[
\mc L=\d^1_0(\mc N)\co \CG\rto \CG.
\]

\begin{prop}\label{P-L-projection}
We have $\mc L^0(\msf e)=\bfe$.
Let $\g=\mc L^0(\f)$. Then $\mc N(\g)$ is identified with $1_\g$.
\end{prop}
\begin{proof} The first statement is straightforward.
Now we show the second statement.
Suppose $\f=(\f^\ell,\SZ,\f^r)$, and
\[
\g=\mc L^0(\f)=(\g^\ell,\SY,\g^r),\;\;\;
\h=\mc L^0(\g)=(\h^\ell, \msf W,\h^r).
\]
Then $Y^0$ consists of equivalence classes $[a,z,b]_\Y$. $W^0$ consists of equivalence classes $[a',[a,z,b]_\SY, b']_{\sf W}$. Here $[]_\Y$ and $[]_{\sf W}$ denote the equivalence class in $Y^0$ and $W^0$. Then $W^0$ has a canonical identification with $Y^0$:
\[
[a',[a,z,b]_\Y,b']_{\msf W} \mapsto [a'a, z,bb']_\Y.
\]
Under this identification, $\g=\h$, and $\mc N(\g)=1_\g$.
\end{proof}

Now we use $\mc N$ to define a natural $\triangleleft$-transformation
\begin{equation*}
\SN\co \SB\CG\rto (\SB\CG)^{[1]}
\end{equation*}
whose source is the identity of $\SB\CG$. The map $N_1\co (\SB\CG)_1\rto (\SB\CG)^{[1]}_{1}$ is defined by
\[
N_1(\f)= \pi_{1,1}^\ast (\mc N(\f)).
\]
Here $\mc N(\f)$ is an element in $\mc G^1$ and is identified with a 2-simplex of $(\SB\CG)_2$, $\pi_{1,1}$ is defined in \eqref{e-pi-n1}.

\begin{coro}\label{C-L} Set $\msf L=\d^1_0\circ\SN$. Then
$\msf L\co (\SB\mc G,\bfe )\to (\SB\mc G,\bfe)$
is a 2-group homomorphism.
\end{coro}
\begin{proof} By Proposition \ref{P-L-projection},  $N_1(\bfe)=\bfe$. Then the claim
 follows from Corollary \ref{C-transformation}.
\end{proof}

\section{2-group actions on Lie groupoids}
\label{Sec-2p-action}

Let $\CX$ be a Lie groupoid and $\SG$ be a 2-group. Here we may treat $\SG$ as a Lie 2-group with discrete topology.
As the automorphisms of $\CX$ can be equipped with a 2-group $\SB\CG:=\SB\SAut(\CX)$, it is natural to think that $\SG$-action on $\mc X$, first of all, is given by a 2-group homomorphism from $\SG$ to $\SB\CG$. Recall that $\SB\CG$ contains the unit $\msf e$, a unit element $\bfe$ and the transformation $\omega_\bullet$.
Set
\begin{equation}
\CH_\SG(\CX)L=\{\rho\co \SG\rto (\SB\CG, \bfe)\mid \rho \mbox{ is a 2-group homomorphism.}\}.
\end{equation}

\begin{defn}\label{D-equiv-H}
We say $\rho_1,\rho_2\in \CH_\SG(\CX)$ are equivalent if
there exists a natural transformation $\SN\co \SG\rto (\SB\CG)^{[1]}$ such that $\d^1_1(\SN)=\rho_1$
and $\d^1_0(\SN)=\rho_2$. Set
\[
\mathbf H_\SG(\CX):=\CH_\SG(\CX)/\sim.
\]
\end{defn}

Let $\CK_\SG(\SX)$ be the set of Kan fibration over $\SG$ with fiber $\SX$, the Lie 1-groupoid corresponding to $\CX$.

\begin{defn}\label{D-equiv-K}
We say two Kan fibrations $\SK_1,\SK_2\in \CK_\SG(\SX)$ are equivalent if there exists a $\bullet$-isomorphism (cf. Definition \ref{D-KF-homomorphism})
$\SF\co \SK_1\rto \SK_2$ of Kan fibrations. Set
\[
\mathbf K_\SG(\SX):=\CK_\SG(\SX)/\sim.
\]
\end{defn}

We may also denote $\CK_\SG(\SX)$ by $\CK_\SG(\CX)$.

The main theorem of this section is

\begin{theorem}\label{T-HK-equivalence}
There is a bijection between
$\mathbf H_\SG(\CX)$ and $\mathbf K_\SG(\CX)$.
\end{theorem}

The proof is given in \S\ref{Sec-K-vs-H}. This theorem leads to the following definition.
\begin{defn}\label{D-2gp-action}
Let $\CX$ be a Lie groupoid and $\SG$ be a 2-group. The following two definitions are equivalent:
\begin{enumerate}
\item A $\SG$-action on $\CX$ is an equivalence class of 2-group homomorphisms in $\mathbf H_\SG(\CX)$.
\item A $\SG$-action on $\CX$ is an equivalence class of Kan fibrations in $\mathbf K_\SG(\CX)$.
\end{enumerate}
\end{defn}

Then we consider general Lie 2-group $\SG$-actions on $\CX$. We use Kan fibration in (2) of Definition \ref{D-2gp-action} to give the definition of smooth $\SG$-actions on Lie groupoids (cf. Definition \ref{D-Kan-smoothaction}).

\subsection{$\mathbf K_\SG(\CX)$ versus $\mathbf H_\SG(\CX)$}\label{Sec-K-vs-H}

\subsubsection{From $\mathbf K_\SG(\CX)$ to $\mathbf H_\SG(\CX)$ }

We first show that a Kan fibration  naturally determines  a homomorphism from $\SG$ to $\SB\SAut(\CX)$.

\begin{prop} Let $\CX$ be a Lie groupoid and $\SG$ be a 2-group.
Given a Kan fibration $\pi\co \SX\rightarrowtail \SK\twoheadrightarrow\SG$, one can associate to it a homomorphism $\rho\co \SG\rto(\SB\CG, \bfe)$. In particular, this yields a map
\begin{equation}
\Lambda\co \mc K_\SG(\mc X)\rto \mc H_\SG(\mc X).
\end{equation}
\end{prop}
\begin{proof} {\em Step 1: description of $\rho_1$.} Let $g\in G_1$, we describe $\u:=\rho_1(g)$. Suppose $\u$ is denoted by
\[
\u\co\CX\xleftarrow{\u^\ell} \mc Y \xrightarrow{\u^r}\CX.
\]
Set
\[
Y^0:=\pi^{-1}_1(g), \qq
u^{\ell,0}:=d^1_1, \and
u^{r,0}:=d^1_0.
\]
Here $d^1_i$ are face maps of $\SK$. Consider $\pi^{[1]}\co \SK^{[1]}\rto \SG^{[1]}$. Set
\[
Y^1:=(\pi^{[1]}_1)^{-1}(\s^0_0(g)),\qq
u^{\ell,1}:=d^{[1],1}_1, \and
u^{r,1}:=d^{[1],1}_0.
\]
{\em Step 2: description of $\rho_2$.}  Let $g_{012}\in G_2$, suppose $\u_{ij}:=\rho_1(g_{ij})$ are
\[
\u_{ij}\co \CX\xleftarrow{\u^\ell_{ij}} \mc Y_{ij} \xrightarrow{\u^r_{ij}}\CX.
\]
The $\rho_2(g_{012})$ is supposed to be given by
a natural transformation from $\u_{01}\star \u_{12}$ to $\u_{02}$. The natural transformation, denoted by
\[
\phi\co  Y^0_{02}\times_{u^{\ell,0}_{02},X^0,u^{\ell,0}_{01}}Y^0_{01}\times_{u^{r,0}_{01},X^0,u^{\ell,0}_{12}} Y^0_{12}\to X^1,
\]
is defined as the following. Let $(y'_{02},y_{01},y_{12})$ be a point of the domain, its image is supposed to be
an arrow from $x_2=u^{r,0}_{12}(y_{12})$ to $x'_2=u^{r,0}_{02}(y'_{02})$. Note that
\[
x_2=u^{r,0}_{12}(y_{12})=\d^1_0(y_{12})=:y_2 \and x'_2=u^{r,0}_{02}(y'_{02})=\d^1_0(y'_{02})=:y'_2.
\]
By the Kan fibration condition, there exists a unique point $y_{012}\in \pi^{-1}_2(g_{012})$ such that its boundary is $(y_{01},y_{12},y_{02})$. Now consider $g'_{012}=s^1_1(g_{02})$, i.e.,
\[
g'_{01}=g'_{02}=g_{02}\and g'_{12}=e.
\]
By the Kan fibration condition, there exists a unique $z_{012}$ over $g'_{012}$ such that
\[
z_{01}=y_{02}\and z_{02}=y'_{02}.
\]
Then $z_{12}\in X^1$ is an arrow from $y_2$ to $y'_2$. Set
\begin{equation}
\phi(y'_{02},y_{01},y_{12}):=z_{12}.
\end{equation}

Once $\rho_2$ is constructed. By Kan conditions, $\rho_n$ for $n\geqslant 3$ are determined by $\rho_2$.  Hence, $\rho$ is constructed.
\end{proof}

\begin{prop}
Let $\SK\sim\tilde \SK$ be equivalent Kan fibrations in $\mc K_\SG(\CX)$, i.e., there exists a $\bullet$-isomorphism
$\SF\co \SK\rto \tilde \SK$. Then $\Lambda(\SK)\sim \Lambda(\tilde\SK)$. Hence, $\Lambda$ induces a map
\[
\mathbf{\Lambda}\co\mathbf{K}_\SG(\CX)\rto \mathbf{H}_\SG(\CX).
\]
\end{prop}
\begin{proof}
Let $\rho=\Lambda(\SK)$ and
$\tilde\rho=\Lambda(\tilde \SK)$. We construct a natural transformation $\SN\co \SG\rto (\SB\CG)^{[1],\bfe}$ from $\rho$ to $\tilde\rho$.

We first construct a map $\mc N\co G_1\rto\Aut^1(\CX)$ such that
\[
s(\mc N(g))=\rho_1(g) \and t(\mc N(g))=\tilde \rho_1(g).
\]
Suppose $\rho_1(g)$ and $\tilde\rho_1(g)$ are
\[
\u\co \CX\xleftarrow{\u^\ell} \mc Y\xrightarrow{\u^r}\CX \and \tilde \u\co \CX\xleftarrow{\tilde\u^{\ell}} \tilde{\mc Y}\xrightarrow{\tilde\u^{r}}\CX.
\]
Then $\mc N(g)$ is supposed to be a map
\begin{equation}
\mc N(g)\co Y^0\times_{u^{\ell,0},X^0,\tilde u^{\ell,0}} \tilde Y^0\rto X^1.
\end{equation}
By the construction of $\Lambda$,
\[
Y^0=\pi^{-1}_1(g)=: K_g, \;\;\;
\tilde Y^0=\tilde\pi^{-1}_1(g):=\tilde K_g.
\]
Also $\SF$ induces an isomorphism
$\mc F_g\co \mc Y\to \tilde {\mc Y}$.  We are ready to define $\mc N(g)$. Let $(y,\tilde y)$ be a point in the domain of $\mc N(g)$, its image is expected to be an arrow from $u^{r,0}(y)$ to $\tilde u^{r,0}(\tilde y)$, which is constructed as the following. Since
\[
\tilde u^{\ell,0}(F^0_g(y))=u^{\ell,0}(y)
=\tilde u^{\ell,0}(\tilde y).
\]
There exists a unique arrow $F(y)\xrightarrow{\alpha} \tilde y$ in the kernel of $\tilde \u^\ell$.
We verify the source and target of $\tilde u^{r,1}(\alpha)$:
\[
s(\tilde u^{r,1}(\alpha))
=\tilde u^{r,0}(\tilde s(\alpha))
=\tilde u^{r,0}(F^0_g(y))=u^{r,0}(y);
\]
\[
t(\tilde u^{r,1}(\alpha))
=\tilde u^{r,0}(\tilde t(\alpha))
=\tilde u^{r,0}(\tilde y).\]
Set
$\tilde u^{r,1}(\alpha)$
to be the desired arrow $\mc N(g)(y,\tilde y)$.

Since $\mc N(g)$ defines a natural transformation from $\rho_1(g)$ to $\tilde\rho_1(g)$, it also determines a 2-simplex in $\SB\CG$ with
\[
(\mc N(g))_{01}=\rho_1(g), \qq
(\mc N(g))_{12}=\msf e,\and
(\mc N(g))_{02}=\tilde \rho_1(g).
\]
Set
\[
N'_1\co G_1\rto (\SB\CG)^{[1]}_{1,\msf e};
\qq
g\mapsto \pi_{2,1}^\ast (\mc N(g)).
\]
Here $N'_1(g)$ satisfies
\[
d^1_0(N'_1(g))=d^1_1(N'_1(g))=\msf e,\;\;\;
\d^1_1(N'_1(g))=\rho(g),\;\;\;
\d^1_0(N'_1(g))=\tilde \rho(g).
\]
Finally, by the Remark \ref{R-deform-square} below, we deform $N'_1(g)$ to be a $\bfe$-transformation $N_1(g)$. This completes the construction of $N_1$, hence of $\SN$.
\end{proof}

\begin{remark}\label{R-deform-square}
Let $\SG$ be a 2-group and $\omega\in G^{[1]}_1$ be a transformation from $e$ to $e'$.
Given an $e$-transformation $\Omega$ from $g$ to $\tilde g$, we may deform it to be a $e'$-transformation $\Omega'$ from $g$ to $\tilde g$. Let $\Delta^{(i)}[1],i=1,2,3$ be three copies of $\Delta[1]$.
Consider an element
\[
\Theta\co \Delta^{(1)}[1]\times \Delta^{(2)}[1]\times \Delta^{(3)}[1]\to \SG.
\]
The $d^{(i),1}_j(\Theta)$ has an obvious meaning. For example
\[
d^{(2),1}_0(\Theta)=\Theta\co \Delta^{(1)}[1]\times \{1\} \times \Delta^{(3)}[1]\to \SG.
\]
By $\Kan!(3)$ of $\SG$, $\Theta$ is uniquely determined by:
\[
d^{(3),1}_1(\Theta), \qq d^{(1),1}_0(\Theta),\qq
d^{(1),1}_1(\Theta), \qq d^{(2),1}_0(\Theta),\and d^{(2),1}_1(\Theta).
\]
We set them to be the following:
\begin{enumerate}
\item $d^{(3),1}_1(\Theta)=\Omega$. Then
$\Omega$ is identified as
\[
\Omega\co \Delta^{(1)}[1]\times \Delta^{(2)}[1]\times\{0\}\to \SG.
\]
Furthermore,
\[
d^1_i(\Omega)=d^{(1),1}_i(\Omega)=e, \qq
\d^1_1(\Omega)=d^{(2),1}_1(\Omega)=g,\and
\d^1_0(\Omega)=d^{(2),1}_0(\Omega)=\tilde g.
\]
\item set $d^{(1),1}_i(\Theta)=\omega$;
\item set $d^{(2),1}_1(\Theta)=\s^0_0(g)$ and $d^{(2),1}_0(\Theta)=\s^0_0(\tilde g)$.
\end{enumerate}
Set $\Omega':=d^{(3),1}_0(\Theta)$.
\end{remark}

\subsubsection{From $\mathbf H_\SG(\CX)$ to $\mathbf K_\SG(\CX)$}

In this subsection, our main goal is to associate a homomorphism $\rho\in \mc H_\SG(\CX)$ with a Kan fibration $\SK_\rho\in \mc K_\SG(\CX)$.
For this purpose,  we construct a Kan fibration over $\SG=(\SB\CG,\bfe)$, which we call the universal Kan fibration of $\CX$:
\begin{equation}
\pi_{\uni}\co\SX\rightarrowtail \SK_{\uni}\twoheadrightarrow \SB\CG.
\end{equation}
The construction is based on the homomorphism $\sf L$ (cf. Corollary \ref{C-L}).
\vskip 0.1in\noindent
{\em Description of $K_{\uni,n}$:}
\begin{enumerate}
\item[(1)] {Description of $K_{\uni,0}$}. Set $K_{\uni,0}:=X^0$.
\item[(2)] {Description of $K_{\uni,1}$}. Let $\f_{01}\in (\SB\CG)_1$. Suppose $L_1(\f_{01})=(\g^\ell_{01},\mc Y_{01},\g^r_{01}).$ Set the fiber $K_{\uni,\f_{01}}$ of $\SK_{\uni}$ over $\f_{01}$ to be $Y^0_{01}$.
\item[(3)] {Structure maps $d^1_j$ and $s^0_0$.} The (restrictions of) face maps $d^1_j$ on  $K_{\uni,\f_{01}}$ is
    \[
    d^1_1=g^{\ell,0}_{01}\and d^1_0=g^{r,0}_{01}.
    \]
    The degeneracy map $s^0_0\co K_{\uni,0}\rto K_{\uni,\bfe}$ is given by
    \[
    s^0_0(x)=[1_x, x, 1_x].
    \]
\item[(4)] {Description of $K_{\uni,2}$}. For a 2-simplex of $\f_{012}$ of $\SB\CG$, denote the fiber of $\SK_\uni$ over $\f_{012}$ by $K_{\uni,\f_{012}}$. Suppose
\[
\f_{ij}=(\f^\ell_{ij},\SZ_{ij},\f^{r}_{ij})\and
L_1(\f_{ij})=\g_{ij}=(\g^\ell_{ij},\SY_{ij},\g^r_{ij}).
\]
By the definition of $\f_{012}$, it determines a natural transformation from
\[
f_{012}\co \f_{01}\star \f_{12}\Rightarrow \f_{02},
\]
which is a map
\[
f_{012}\co Z^{0}_{02}\times_{f^{\ell,0}_{02}, X^0, f^{\ell,0}_{01}}
Z^0_{01}\times_{f^{r,0}_{01},X^0, f^{\ell,0}_{12}} Z^0_{12}\rto X^1.
\]
Then
\[
K_{\uni,\f_{012}}\subseteq \Hom(\partial\Delta[2],\sk_1(\SK_\uni))\cap \left(
K_{\uni,\f_{12}}\times
K_{\uni,\f_{02}}\times K_{\uni,\f_{01}}\right).
\]
Take an element of the set on the right hand side
\[
(x_{12},x_{02},x_{01}), \qq \mbox{ where } \qq x_{ij}=[a_{ij},z_{ij},b_{ij}]\in K_{\uni,\f_{ij}}=Y^0_{ij}.
\]
For simplicity of the expression, we choose particular representatives of $x_{ij}$:
\begin{itemize}
\item for $x_{01}$, take $b_{01}$ to be the unit arrow, i.e,
\[
x_{01}=[a_{01},z_{01}, 1_x], \mbox{ where } x=f^{r,0}_{01}(z_{01});
\]
\item for $x_{12}$, take $a_{12}$ to be the unit arrow, i.e,
\[
x_{12}=[1_x, z_{12},b_{12}];
\]
\item for $x_{02}$, take $z_{02}$ such that
\[
f^{\ell,0}_{01}(z_{01})
=f^{\ell,0}_{02}(z_{02}),
\]
and take $a_{02}$ to be $a_{01}$, i.e,
\[
x_{02}=[a_{01},z_{02},b_{02}].
\]
\end{itemize}
With these choices,
\begin{equation}
(x_{12},x_{02},x_{01})\in
K_{\uni,\f_{012}}
\iff f_{012}(z_{01},z_{12},z_{02})\cdot b_{02}=b_{12}.
\end{equation}
\item[(5)] {Description of structure maps $d^2_j$ and $s^1_i$.} Suppose
$x_{012}=(x_{12},x_{02},x_{01})$, then
\[
d^2_0(x_{012})=x_{12},\qq
d^2_1(x_{012})=x_{02},\and
d^2_2(x_{012})=x_{01}.
\]
The degeneracy maps are given by
\begin{align*}
s^1_0([1_{f^{\ell,0}(z)},z,b])&= ([1_{f^{\ell,0}(z)},z,b],[1_{f^{\ell,0}(z)},z,b], [1_{f^{\ell,0}(z)},f^{\ell,0}(z),1_{f^{\ell,0}(z)}]),\\
s^1_1([a,z,1_{f^{r,0}(z)}])&=([1_{f^{r,0}(z)},f^{r,0}(z),1_{f^{r,0}(z)}], [a,z,1_{f^{r,0}(z)}], [a,z,1_{f^{r,0}(z)}]).
\end{align*}
\end{enumerate}
The $K_{\uni,i}$ for $0\leqslant i\leqslant 2$ are enough to generate a Kan fibration $\SK_\uni$ over $\SB\CG$. This completes the construction of $\SK_\uni$. We omit the details  verifying that $\SK_\uni$ is a Kan fibration over $\SB\CG$.

\begin{prop}
Given a homomorphism $\rho\co \SG\rto (\SB\SAut(\mc X),\bfe)$, one can associate to it a Kan fibration
\[
\pi_\rho\co \mc X\rightarrowtail \msf K_\rho \twoheadrightarrow \SG.
\]
This yields a map
\begin{equation}
\Pi\co \mc H_\SG(\mc X)\rto \mc K_\SG(\mc X).
\end{equation}
\end{prop}
\begin{proof}
Set $\msf K_\rho=\rho^\ast \msf K_{\uni}$.
\end{proof}

\begin{prop}
If $\rho\sim \tilde\rho$ are equivalent homomorphisms in $\CH_\SG(\CX)$, then $\Pi(\rho)$ and $\Pi(\tilde \rho)$ are equivalent Kan fibrations in $\CK_\SG(\CX)$. Therefore $\Pi$ induces a map
\begin{equation}
\mathbf{\Pi}\co\mathbf H_\SG(\CX)\rto \mathbf K_\SG(\CX).
\end{equation}
\end{prop}
\begin{proof} Let $\SN\co \SG\rto (\SB\CG)^{[1]}$ be the natural transformation from $\rho$ to $\tilde \rho$.
Hence, for each $g\in G_1$, $N_1(g)$ is a $\bfe$-transformation, i.e,
\[
d^1_i(N_1(g))=\bfe,\qq\text{for}\qq i=0,1,
\]
and
\[
\d^1_1(N_1(g))=\rho_1(g) \and
\d^1_0(N_1(g))=\tilde\rho_1(g).
\]

Let $\SK$ and $\tilde \SK$ denote Kan fibrations $\Pi(\rho)$ and $\Pi(\tilde\rho)$ respectively.
By the definition of $\Pi$, we have $\bullet$-homomorphisms of Kan fibrations:
\[
\SP\co \SK\rto\SK_\uni \and
\tilde \SP\co \tilde \SK\rto \SK_\uni.
\]
We next construct a homomorphism $\SQ\co \SK\to \SK_{\mathrm{uni}}^{[1]}$ that covers $\SN\co \SG\to(\SB\CG)^{[1]}$ such that it is a natural transformation from $\SP$ to $\tilde \SP$.
\vskip 0.1in
\noindent
{\em Definition of $\SQ$:}
\vskip 0.1in
\begin{enumerate}
\item Description of $Q_0$.
As $N_0(\bullet)=\bfe$, $Q_0$ is a map from  $K_\bullet=X^0$ to $K_{\uni,\bfe}=X^1$; set $Q_0$ to be the unit of $\CX$, or equivalently, $s^0_{\uni,0}$ of $\SK_{\uni}$.
\item Description of $Q_1$. We describe $Q_1(y)$ for $y\in K_1$. Suppose
\[
\pi_{\rho,1}(y)=g,\qq
d^1_{\rho,1}(y)=x, \and d^1_{\rho,0}(y)=x'.
\]
The $Q_1(y)$ is an element in $K^{[1]}_{\uni,1}$ which is over $N_1(g)$, by the Kan fibration condition of $\SK_\uni$, $Q_1(y)$ is determined by
\[
d^1_{\uni,1}(Q_1(y))=s^0_{\uni,0}(x),\qq
d^1_{\uni,0}(Q_1(y))=s^0_{\uni,0}(x'),\and
\d^1_{\uni,1}(Q_1(y))=P_1(y).\]
\item Description of $Q_n,n\geq 2$. This follows from $Q_1$ and $\Kan!(3)$ of $\SK_\uni$.
\end{enumerate}
The $\SQ$ then induces a $\bullet$-isomorphism from
\[
\SF\co \SK\rto \tilde\SK
\]
by $\SN$. We explain the definition of $F_1$, others are similar. Let $y\in K_1$, by the construction of $Q_1(y)$,
\[
\d^1_{\uni,1}(Q_1(y))=P_1(y) \and
\d^1_{\uni,0}(Q_1(y))=\tilde P_1(\tilde y)
\]
for some $\tilde y\in \tilde K_1$. Since $\tilde P_1$ is invertible, we get a correspondence $y\mapsto \tilde y$. This defines $F_1$. It is easy to see that $y$ and $\tilde y$ are determined by each other, so $F_1$ is invertible.
\end{proof}

\subsubsection{$\mathbf H_\SG(\CX)\cong \mathbf K_\SG(\CX)$}\label{Sec-H=K}

Let $\rho_\uni$ be the homomorphism given by $\SK_\uni$.

\begin{theorem}
By the construction of $\SK_{\uni}$ we have $\rho_\uni=\msf L.$
\end{theorem}

We are now ready to proof Theorem \ref{T-HK-equivalence}.

\begin{proof}[Proof of Theorem \ref{T-HK-equivalence}]
By the construction, it is obvious that $\mathbf \Pi\circ \mathbf \Lambda=id$.

Let $\rho\in \mc H_\SG(\CX)$, then
\[
(\Lambda\circ\Pi)(\rho)=\msf L\circ\rho.
\]

Since $\msf L\Rightarrow \msf{id}\co \SB\CG\rto \SB\CG$,
then $\msf L\circ$ is natural transformed to
$\msf{id}\circ\co \CH_\SG(\CX)\rto \CH_\SG(\CX)$ which is $id \co \CH_\SG(\CX)\rto \CH_\SG(\CX)$, hence,
\begin{equation}
\Lambda\circ \Pi\Rightarrow id.
\end{equation}
We have $\msf L\circ\rho\sim \rho$. Therefore $\mathbf \Lambda\circ\mathbf \Pi=id$.
\end{proof}

\subsection{Lie 2-group actions on Lie groupoids}

In previous subsection, we explain that a discrete 2-group $\SG$-action on a Lie groupoid $\CX$ has two equivalent descriptions, via $\mathbf H_\SG(\CX)$ and $\mathbf K_\SG(\CX)$. We may regard the Kan fibration $\Pi(\rho)$ as a geometric realization of the action homomorphism $\rho\co \SG\rto \SB\SAut(\CX)$. This can be seen from the following example.

\begin{example}\label{E-smoothaction-manifold}
Recall that a group $G$-action on a smooth manifold $X$ given  by a group homomorphism $\rho\co G\rto \Diff(X)$ has a geometric realization $\Phi\co G\times X\rto X$. Now we translate them into the language used in \S\ref{Sec-HK}.

In terms of higher groups, $\rho$ corresponds to a 1-group homomorphism
\[
\SB\rho\co \SB G=(G\rrto\bullet)\rto \SB\Diff(X)=(\Diff(X)\rrto \bullet).
\]
The Kan fibration $\Pi(\SB\rho)$ is
\[
\pi\co \SK= (G\times X\rrto X) \rto (G\rrto \bullet),
\]
and the $\Phi$ is the target map of $\SK$.
\end{example}

We give the geometric realization of $\rho$ a name.

\begin{defn}\label{D-quotient}
For $\rho\in \mc H_\SG(\CX)$, we call the Kan fibration $\Pi(\rho)$ the {\bf quotient groupoid} of $\CX$ with respect to $\rho$ and denote it by $\CX\rtimes_\rho \SG$.
\end{defn}

\begin{remark}
Although we only define actions of a Lie 2-group $\SG$ on  a Lie groupoid $\CX$, one may generalize the definition to allow $\SG$ to be a Lie 2-groupoid. Such a generalization is trivial with slight modifications.
\end{remark}

Now we consider the case that $\SG$ is a Lie 2-group. As $\SB\SAut(\CX)$ is not equipped with any smooth structure yet, it is impossible to define smooth action of $\SG$ on $\CX$ in terms of $\mathbf H_\SG(\CX)$, while it can be done in terms of $\mathbf K_\SG(\CX)$.  This leads to the following definition.

\begin{defn}\label{D-Kan-smoothaction}
Let $\CX$ be a Lie groupoid and $\SG$ be a Lie 2-group. By a smooth $\SG$-action on $\CX$, we mean an equivalent class in $\mathbf K_\SG(\CX)$.
\end{defn}

Let $\pi\co\CX\rightarrowtail \SK\twoheadrightarrow \SG$ be a Kan fibration. We may write $\SK=\CX\rtimes \SG$, or $\CX\rtimes_\rho\SG$ for $\rho=\Lambda(\SK)$. On the other hand, giving a $\rho\in \mc H_\SG(\CX)$, we say it induces a smooth action on $\CX$ if the Kan fibration $\SK=\Pi(\rho)$ is a smooth Kan fibration with respect to the smooth structure of $\SG$ (rather than the discrete topology).

We discuss some important concepts related to $\SG$-actions.

\vskip 0.1in\noindent
{\bf $\SG$-invariant groupoid.}
\vskip 0.1in

Let $\pi\co \CX\rightarrowtail \SK\twoheadrightarrow \SG$ be a Kan fibration. Note that $K_0=X^0$. It is natural to consider $\SG$-invariant sub-groupoid of the action.

Let $x\in X^0$. Let
\[
\Gamma^\SK_{x,1}=\{y\in K_1 \mid d^1_i(y)=x, i=0,1\}
\]
be the $\SK$-isotropy set of $x$. We say $x$ is {\bf $\SG$-invariant} if $\pi_1\co \Gamma^\SK_{1,x}\rto G_1$ is surjective. Let $(X^0)^\SG\subseteq X^0$ be the subspace of $\SG$-invariant points. One can show that this is an $X^1$-invariant space, i.e., for
any arrow $x\xrightarrow{\alpha}y'$ with $x$ being $\SG$-invariant, $x'$ is also $\SG$-invariant. Set $\CX^\SG:=(X^0)^\SG\rtimes \CX$. We remark that $(X^0)^\SG$ usually is not a smooth manifold and $\CX^\SG$ is not a Lie groupoid.

\vskip 0.1in
\noindent
{\bf $\SG$-equivariant homomorphisms and equivalent actions}
\vskip 0.1in

Let $\pi_i\co\SX_i\rightarrowtail \msf K_{i}\twoheadrightarrow \SG, i=1,2$ be two Kan fibrations. Suppose $\msf F\co \msf K_1\rto \msf K_2$ is a homomorphism of Kan fibrations. Then we say that the restriction of $\msf F$ on fibers, denoted by
$\f\co\SX_1\rto \SX_2$, is a {\bf $\SG$-equivariant homomorphism}. If $\f$ is an equivalence, we say the two 2-group actions given by Kan fibrations are equivalent. This is illustrated by the following diagram.
\begin{equation}\label{e-Kan-homomorphism}
\begin{tikzcd}
\SX_1 \arrow[d, "\msf f"'] \arrow[r] & \SK_1 \arrow[r, "\pi_1"] \arrow[d, "\msf F"'] & \SG \arrow[d, "\cong"] \\
\SX_2 \arrow[r]                      & \SK_2 \arrow[r, "\pi_2"]                      & \SG.
\end{tikzcd}
\end{equation}
\vskip 0.1in
\noindent
{\bf $\SG$-equivariant vector bundles and sections}
\vskip 0.1in
Consider the diagram \eqref{e-Kan-homomorphism}. If $\sf F\co K_1\rto K_2 $ is a vector bundle, and structure maps of $\SK_1$ are compatible with vector bundle structures, we say $\SK_1$ is a {\bf Kan-vector bundle} over $\SK_2$. Then $\f\co \SX_1\rto \SX_2$ is called a {\bf $\SG$-equivariant vector bundle}. If a section $\sigma_\SG$ of $\SF\co \SK_1\rto \SK_2$ is a homomorphism of Kan fibrations, we say $\sigma_\SG$ is a {\bf Kan-section}. Let $\sigma$ be the restriction of $\sigma_\SG$ on fibers, $\sigma$ is called an {\bf $\SG$-equivariant section} of the bundle $\f\co \SX_1\rto\SX_2$. Let $\Gamma_\SG(\SK_2,\SK_1)$ be the vector space of Kan-sections and $\Gamma_\SG(\SX_2,\SX_1)$ be the vector space of $\SG$-equivariant sections. Then
\begin{equation}
\Gamma_\SG(\SK_2,\SK_1)\cong
\Gamma_\SG(\SX_2,\SX_1).
\end{equation}

\vskip 0.1in
\noindent
{\bf Induced actions and the pushed-forward of Kan fibrations}
\vskip 0.1in

Let $\phi\co \SH\to \SG$ be a homomorphism of Lie 2-groups. Let
\[
\pi\co \X\rightarrowtail \SK \twoheadrightarrow \SG
\]
be the Kan fibration of a $\SG$-action on $\SX$. Then $\phi^\ast \SK$ is a Kan fibration
\[
\pi'\co \X\rightarrowtail \phi^\ast\SK \twoheadrightarrow \SH
\]
which induces an $\SH$-action on $\SX$.

Now we consider a morphism $\phi\co \SH\rightharpoonup \SG$ given by
\[
\phi\co \SH\xleftarrow{\phi^\ell} \tilde{\SH} \xrightarrow{\phi^r} \SG,
\]
where $\phi^\ell$ is a hypercover. Then we have a pulled-back Kan fibration
\[
\tilde \pi\co \X\rightarrowtail
\phi^{r,\ast}\SK \twoheadrightarrow \tilde \SH.
\]
In general, we do not have push-forward for Kan fibrations. Nevertheless, we have the following result.

\begin{prop}\label{P push-forward}
Let $\f\co\SX\rto\SY$ be a hypercover of Lie $2$-groupoids with $f_0$ being a diffeomorphism and $\pi\co \SE\rto \SX$ be a Kan fibration over $\SX$. Then there exists a Lie $2$-groupoid Kan fibration $\f_\ast\SE\rto\SY$ in the sense that $\f^\ast(\f_\ast\SE)\cong\SE$. In particular, this applies to the case that $\SX$ and $\SY$ are both Lie 2-groups.
\end{prop}
\begin{proof}
For simplicity, we denote the desired Lie 2-groupoid $\f_\ast\SE$ by $\SF$. The strategy of constructing $F_n$ can be described as the following. The surjective submersion map $f_n\co  X_n\rto Y_n$ defines an equivalence relation on $X_n$, which is interpreted by the Lie groupoid
\[
\CX_n=(X_n\times_{f_n,Y_n,f_n}X_n\rrto X_n),
\]
called a submersion groupoid in the literature (see for example \cite[Example 3.2.3]{Hoyo2013}), whose source and target maps are $pr_1$ and $pr_2$. The coarse space $|\CX_n|=Y_n$. The coarse map is just $f_n$. The goal is to describe a natural $\CX_n$ action on $\pi_n\co E_n\rto X_n$, i.e., a fibration $\CF_n=E_n\rtimes \CX_n$ over $\CX_n$. Then $F_n$, defined as the coarse space $|\CF_n|$, is a fiber bundle over $Y_n$.

To achieve the goal, we identify $\CX^1_n=X_n\times_{f_n,Y_n,f_n}X_n$ with a subspace $\CV^1_n$ of $X^{[1]}_n$. The Lie groupoid $\CX_n$ then yields a Lie groupoid $\CV_n=(\CV^1_n\rrto X_n)$. The structure maps of $\CV_n$ then can be interpreted as some canonical maps between $\SX$ and $\SX^{[1]}$.

The constructions of $\CF_n$ and $\CV_n$ can be done inductively on $n$. We need the projection homomorphisms
\begin{align}\label{E map-p-n}
\msf p_n\co \Delta[1]\times \Delta[n]\rto \Delta[n],
\end{align}
which maps $\Delta^\bullet_1\times\Delta^\bullet_n$
to $\Delta^\bullet_n$. Set $W_n^1\subseteq Y^{[1]}_n$ be the subspace
\[
W_n^1:=\left\{\left.\msf p_n^\ast(\beta_n)\in Y^{[1]}_n\,\right|\, \beta_n\in Y_n\right\}\cong Y_n.
\]
It is clear that $d^n_j(W^1_n)=W^1_{n-1}$.

By the mathematical induction, for each $n\geqslant 0$,
we will construct a fiber bundle of Lie groupoids
\[
q_n\co \CF_n=(\CF^1_n\rrto E_n)\rto \CV_n=(\CV^1_n\rrto X_n),
\]
and an isomorphism $u_n=(u_n^1,id_{X_n})\co \CV_n\rto \CX_n$.

We first explain the construction for $n=0$. Let $\CV^1_0\subseteq (f^{[1]}_0)^{-1}(W^1_0)\subseteq X^{[1]}_0$ be the subspace
\[
\CV^1_0:=\left\{\left.\msf p^\ast_0(\alpha_0)\in X_0^{[1]}\,\right|\,\alpha_0\in X_0\right\} \cong X_0,
\]
and $u^1_0\co \CV^1_0\rto \CX^1_0$ be the natural isomorphism. Let $\CF^1_0\subseteq (\pi^{[1]}_0)^{-1}(V^1_0)$ be the subspace
\[
\CF^1_0:=\left\{\left.\msf p^\ast_0(z_0)\in E_0^{[1]}\,\right|\, z_{0}\in E_0\right\} \cong E_0.
\]
We have the diagrams
\begin{align}\label{D Prop1.25-1}
\begin{split}
\xymatrix{
\CF^1_0 \ar[r]^-{(\d^1_i)_0}\ar[d]_-{\pi^{[1]}_0}  & E_0 \ar[d]^-{\pi_0} \\ \CV^1_0 \ar[r]^-{(\d^1_i)_0} & X_0.}
\end{split}
\end{align}
 They serve as the source and target maps of the following two Lie groupoids
\[
\CF_0=(\CF^1_0\rrto E_0),\qq \text{and}\qq\CV_0=(\CV^1_0\rrto X_0),
\]
such that $u_0=(u^1_0,id_{X_0})\co \CV_0\rto \CX_0$ is an isomorphism of Lie groupoids. Moreover, the diagram \eqref{D Prop1.25-1} is cartesian, i.e., $\CF_0^1$ is naturally diffeomorphic to the pullbacks $(\d^1_i)_0^*E_0$ via $(\pi_0^{[1]},(\d^1_i)_0)$.  Therefore $\tilde\pi_0=(\pi^{[1]}_0,\pi_0)\co \CF_0\rto \CV_0$ is a Lie groupoid morphism such that $\CF_0$ is the action Lie groupoid of the $\CV_0$-action on $E_0$. Set $F_0=|\mc F_0|$, it is a fiber bundle over $Y_0=|\CV_0|$. Obviously, $(F_0\rto Y_0)\cong (E_0\rto X_0)$.

As the mathematical induction assumption, we suppose
\begin{itemize}
\item $\CF_{n-1}=(\CF^1_{n-1}\rrto E_{n-1})$, $\CV_{n-1}=(\CV^1_{n-1}\rrto X_{n-1})$ with source and target maps being given by $(\d^1_i)_{n-1}$, and
\item the isomorphism $u_{n-1}=(id_{X_{n-1}},u^1_{n-1})\co \CV_{n-1}\rto \CX_{n-1}$
\end{itemize}
are constructed, such that $\tilde\pi_{n-1}=(\pi^{[1]}_{n-1},\pi_{n-1})\co \CF_{n-1}\rto \CV_{n-1}$
is the action Lie groupoid of the $\CV_{n-1}$-action on $E_{n-1}$, i.e., we have the cartesian diagram
\begin{align}\label{D Prop1.25-2}
\begin{split}
\xymatrix{
\CF^1_{n-1} \ar[r]^-{(\d^1_i)_{n-1}} \ar[d]_-{\pi^{[1]}_{n-1}}  & E_{n-1} \ar[d]^-{\pi_{n-1}}\\
\CV^1_{n-1} \ar[r]^-{(\d^1_i)_{n-1}} & X_{n-1}.}
\end{split}
\end{align}
We next construct $\CV_n$, $\CF_n$, the isomorphism $u_n=(id_{X_n},u_n^1)\co \CV_n\rto\CX_n$ and prove the $n$-th cartesian diagram as the $(n-1)$-th one in \eqref{D Prop1.25-2}.

Recall that $d^n_j(W^1_n)\cong W^1_{n-1}$. For $f^{[1]}_n\co X^{[1]}_n\rto Y^{[1]}_n$, set $\CV^1_n\subseteq
(f^{[1]}_n)^{-1}(W^1_n)\subseteq X_n^{[1]}$ to be the subspace
\[
\CV^1_n:=\left\{\left.\alpha\in (f^{[1]}_n)^{-1}(W^1_n)\,\right|\,d^n_j (\alpha)\in \CV^1_{n-1}\right\}.
\]
For $\pi^{[1]}_n\co  E^{[1]}_n\rto X^{[1]}_n$, set $\CF^1_n\subseteq
(\pi^{[1]}_n)^{-1}(V^1_n)\subseteq E_n^{[1]}$ to be the subspace
\[
\CF^1_n:=\left\{\left. z\in (\pi^{[1]}_n)^{-1}(V^1_n)\,\right|\, d^n_j (z)\in \CF^1_{n-1}\right\}.
\]
We also have the diagram
\begin{align}\label{D Prop1.25-3}
\begin{split}
\xymatrix{
\CF^1_n \ar[r]^-{(\d^1_i)_n} \ar[d]_-{\pi^{[1]}_n}  & E_n \ar[d]^-{\pi_n}\\
\CV^1_n \ar[r]^-{(\d^1_i)_n} & X_n.}
\end{split}
\end{align}
We define the diffeomorphism $u_n^1\co \CV_n^1\rto \CX_n^1$ by
\[
u^1_n(\alpha)=\left((\d^1_0)_n (\alpha), (\d^1_0)_n\circ f^{[1]}_n(\alpha), (\d^1_1)_n(\alpha)\right).
\]
Then by $\Acyc!(n+1)$ for $n\geqslant 1$ we see that $u_n^1$ is an diffeomorphism. Moreover, via this isomorphism, the groupoid structure of $\CX_n$ induces a groupoid structure over $\CV_n:=(\CV_n^1\rrto X_n)$ with source and target maps being given by $(\d^1_1)_n$ and $(\d^1_0)_n$ respectively. This gives rise to the isomorphism $u_n=(u_n^1,id_{X_n})\co \CV_n\rto\CX_n$. Finally, by $\Kan!(n+1)(\pi)$ for $n\geqslant 1$ of the Kan fibration $\SE\rto\SX$ of Lie $2$-groupoids, we have the diffeomorphisms
\[
((\d^1_i)_n,\pi_n^{[1]})\co \CF_n^1\rto E_n\times_{\pi_n,X_n,(\d^1_i)_n}\CV_n^1,\qq\forall\, i=0,1.
\]
This shows that the diagram \eqref{D Prop1.25-3} is cartesian. Therefore
\[
\tilde\pi_n=(\pi^{[1]}_n,\pi_n)\co \CF_n=(\CF^1_n\rrto E_n)\rto \CV_n=(\CV^1_n\rrto X_n)
\]
is a fiber bundle of Lie groupoids corresponding to the $\CV_n$-action on $E_n$.

So by the mathematical induction, we have constructed
\begin{description}
\item[(1)] an isomorphism $u_n\co \CV_n\rto \CX_n$;
\item[(2)] a fiber bundle $\tilde\pi_n\co \CF_n\rto \CV_n$.
\end{description}
Set $F_n=|\CF_n|$, then it is a fiber bundle over
$Y_n=|\CX_n|$. With structure maps induced from $\SE$, $F_n$ are integrated into a Lie 2-groupoid $\SF$. Moreover the $\Kan(1)(\pi)$ and $\Kan!(n)(\pi)$ for $n\geqslant 2$ of $\SE\rto\SX$ also implies that $\SF\rto \SY$ is a Kan fibration of Lie $2$-groupoids. Finally, by the construction we have $\f^*\SF=\SE$. We leave details to readers.
\end{proof}

\begin{coro}
Let $\phi\co \SH\rightharpoonup
\SG$ be a Lie 2-group morphism, and $\SK$ be a Kan fibration over $\SG$ with fiber $\X$. Then we have
pull-back Kan fibration $\phi^\ast\SK$ over $\SH$ which is defined to be $\phi^{\ell}_\ast\phi^{r,\ast}\SK$.
\end{coro}

\subsection{Group actions on Lie groupoids}\label{Sec-gp-actionon1}

We consider the special situation where a Lie group $G$ acts on a Lie groupoid $\CX=(X^1\rrto X^0)$. By Theorem \ref{T-gpd-gp}, the Lie groupoid $(G\rrto \bullet)$ has an associated Lie 1-group $\SB G$.

Let $\pi\co \CX\rightarrowtail \SK \rto \SB G$ be a Kan fibration of a $\SB G$-action on $\CX$. We remark that this is a Kan fibration of Lie 2-groupoids.

\begin{lemma}
Since $\SB G$ is a Lie 1-group, the $\SK$ is a Lie 1-groupoid.
\end{lemma}

The proof is straightforward. The $\Kan(1)(\pi)$ condition says that the maps
\begin{equation}
(s,\pi)\co K_1\rto K_0\times G \and
(t,\pi)\co K_1\rto K_0\times G
\end{equation}
are surjective submersions. Here $s,t$ are source and target maps of $\SK$.

We consider a special case when $G$ acts strictly on $\CX$ in the following sense.
Suppose the $G$-action on $\CX$ is given by a homomorphism
\begin{equation}
\Phi\co G\times \CX\rto \CX.
\end{equation}
Then $\Phi^i\co G\times X^i\rto X^i, i=0,1$ give $G$-actions on both $X^0$ and $X^1$. The associated Kan fibration is
\begin{equation}
\pi\co \SK=\CX\rtimes\SB G=(X^1\times G\rrto X^0)\rto \SB G.
\end{equation}\label{e-KF-strict}
The structure maps of $\SK$ are
\begin{enumerate}
\item $s_\SK(\alpha, g)=s(\alpha)$,
$t_\SK(\alpha,g)=t(\alpha)\cdot g$;
\item $u(x)=(1_x, 1)$;
\item $(\alpha, g)\cdot (\beta,h)=(\alpha\cdot (\beta\cdot g^{-1}), gh)$;
\item $i(\alpha,g)=(i(\alpha\cdot g), g^{-1})$.
\end{enumerate}

\begin{lemma}\label{L-strict-quotient}
Suppose $G$ acts strictly on $\CX$. If $G$ acts freely on $X^0$, (hence, also on $X^1$), then
\[
\CX/G:=(X^1/G\rrto X^0/G)
\]
is a Lie groupoid. Furthermore, the obvious projection $\SK=\CX\rtimes\SB G\rto \mc X/G$ is a hypercover.
\end{lemma}

The proof is straightforward. We omit the details.
Of course, a general $G$-action on $\mc X$ may not be strict.

\begin{example}
Let $X=S^1$ and $G=S^1$ act on $X$ as rotations. Let $\{U_1,U_2\}$ be an open cover of $S^1$ where $U_i\not=X$. We have a Lie groupoid $\CX$ that is a refinement of $X$ with
\[
X^0=U_1\sqcup U_2 \and X^1=(U_1\cap U_2) \sqcup (U_2\cap U_1).
\]
It is clear that the $G$-action on $X$ can not be realized as a strict $G$-action on $\CX$.
\end{example}

Although the $G$-action on $\CX$ is not strict, if $\CX$ is replaced by $X$, the action is strict. We may ask  whether any $\CX$ has an equivalence replacement such that the induced $G$-action is strict. The answer is positive.

\begin{prop} \label{P-tilde-X}
Let $\pi\co \CX\rightarrowtail\SK\rto \SB G$ be a Kan fibration. There exists a Lie groupoid $\tilde \CX$ with strict $G$-action that is equivalent to the given action  in the following sense:
\[
\begin{tikzcd}
\CX \arrow[d, "\f"'] \arrow[r] & \SK \arrow[r, "\pi"] \arrow[d, "\msf F"'] & \SB G \arrow[d, "\cong"] \\
\tilde\CX \arrow[r]            & \tilde \SK \arrow[r, "\tilde\pi"]         & {\SB G,}
\end{tikzcd}
\]
where $\f$ is an equivalence.
\end{prop}
\begin{proof}
Let $\SE G=(G_0\times G\rrto G_0)$ be the Lie groupoid of
right $G$-action on $G$, where we denote the second $G$ by $G_0$. We denote $\SE G=G_0\rtimes \SB G$. It is a Kan fibration
\[
\pi'\co G_0\rightarrowtail \SE G\twoheadrightarrow  \SB G.
\]
Set
\[
\tilde\CX=\SK\times_{\SB G}\SE G.
\]
It is a Lie groupoid, and the 0 and 1-spaces of $\tilde\CX$ are
\[
\tilde X^0=K^0\times G_0 \and
\tilde X^1=K^1\times_G (G_0\times G)\cong K^1\times G_0.
\]
The source and target maps are
\[
\tilde s(\alpha, g_0)=(s_\SK (\alpha), g_0) \and
\tilde t (\alpha, g_0)=(t_\SK(\alpha), g_0\pi_1(\alpha)).
\]
As the fiber of $\pi$ is $\CX$, we have the embedding $\msf i\co \CX\rto \SK$.

Set
\begin{equation}
f^0\co X^0\xrightarrow{\iota^0} K^0\rto K^0\times \{e_0\}\subset \tilde X^0 \and
f^1\co X^1\xrightarrow{\iota^1} K^1\rto K^1\times \{e_0\}\subset \tilde X^1.
\end{equation}
Here $e_0$ is the unit of $G_0$. This defines a homomorphism of Lie groupoids $\f\co \CX\rto \tilde\CX$. It is straightforward to show that $\f$ is an equivalence.

The groupoid $\tilde\CX$ admits a strict left $G$-action: for any $h\in G$, the action of $h$ is given by the following morphism
\[
\lambda_h\co \tilde\CX\rto \tilde\CX;\qq (x,g)\xrightarrow{\lambda_h^0}
(x,h^{-1}g),\and
(\alpha,g)\xrightarrow{\lambda_h^1}(\alpha,h^{-1}g).
\]
Since the action is free, it is direct to see that $\tilde \CX/G\cong \SK$.

On the other hand, applying \eqref{e-KF-strict} to the strict action of $G$ on $\tilde \CX$, we get the Kan fibration
\begin{equation}
\tilde\pi\co \tilde\CX\rightarrowtail\tilde{\SK}\twoheadrightarrow \SB G,
\end{equation}
where
\begin{equation}
\tilde {\SK}= (G\times K^1\times G_0\rrto K^0\times G_0).
\end{equation}
Define $\msf F\co \SK\rto \tilde{\SK}$ by setting
\[
F^0(m)=(m,e_0) \and  F^1(\alpha)=(\pi^1(\alpha),\alpha, e_0).
\]
One verifies that $\f,\sf F$ fit into the claimed diagram.
\end{proof}

A similar argument yields the slice theorem for group actions on Lie groupoids. We recall the slice theorem for a Lie group $G$-action on a smooth manifold $X$. For simplicity, we assume $G$ and $X$ are compact. For any $x\in X$ with isotropy group $\Gamma_x$, there exists a  $\Gamma_x$-invariant smooth submanifold $S_x$ passing through $x$ together with a diffeomorphism
\[
\psi\co S_x\times_{\Gamma_x} G\rto U
\]
onto a neighborhood $U$ of the orbit $G\cdot x$. Moreover, $\psi$ is $G$-equivariant.

\begin{remark}
It is known that given a proper Lie groupoid $\CX$, for any $x\in X^0$ the  slice exists. Hence, the slice $S_x$ can be interpreted as a slice of the proper Lie groupoid
$G\ltimes X$.
\end{remark}

The diffeomorphism of $\psi$ gives $U\cong S_x\times_{\Gamma_x}G$, which in Lie groupoid terms reads
\[
S_x\times_{\Gamma_x}G\cong (\Gamma_x\times S_x\times G\rrto S_x\times G)=:\Gamma_x\ltimes (S_x\times G).
\]
The diffeomorphism $\psi$ is $G$-equivariant can be read as
\[
U\rtimes G\cong (\Gamma_x\ltimes(S_x\times G))\rtimes G.
\]
In particular, we see that the $G$-action on the Lie groupoid $\Gamma_x\ltimes (S_x\times G)$ is strict and free on both $0$ and $1$-space.

Now we consider the group $G$-action on a Lie groupoid $\CX$. Let $\pi\co \SK\rto \SB G$ be the Kan fibration. It is easy to see that $G$ acts on the coarse space $|\CX|$ of $\CX$ and $|\SK|\cong|\CX|/G$. In fact, the coarse map $|\msf i|$ of the inclusion $\msf i\co \CX\rto\SK$
induces the projection map $|\CX|\to |\SK|.$

Since $\CX$ is compact and $G$ is compact, we can show that $\SK$ is proper. Hence, for any $x\in X^0$, there exists slice \footnote{This follows from the linearization theorem, see \cite{Weinstein2000-linearization, Zung2006, Pflaum-Posthuma-Tang2014, Crainic-Struchiner2013-Linearization,Mendes-Radeschi2018}.} $S_x\subset K^0$ of $\SK$, i.e,
\begin{equation}
\Gamma_x\ltimes S_x \cong \SK_{|S_x|}.
\end{equation}
Here $|S_x|$ is an open neighborhood of $|x|$, $\SK_{|S_x|}$ is the induced Lie groupoid structure on $|S_x|$. Let $U=|\msf i|^{-1}(|S_x|)$. It is clear we have a restriction of $G$-action on $\CX_U$ whose Kan fibration is
\begin{equation}
\pi_U\co \CX_U\rightarrowtail
\SK_{|S_x|}\twoheadrightarrow \SB G.
\end{equation}
Now we apply the construction given in Proposition \ref{P-tilde-X}:
\begin{equation}\label{e-4.18}
\CX_U\simeq \tilde{\CX}_U\cong \SK_{|S_x|}
\times_{\SB G} \SE G
\cong (\Gamma_x\ltimes S_x)\times_{\SB G}\SE G.
\end{equation}
By the restriction of $\pi$, we have the map $\pi: \Gamma_x\rtimes {x}\rto \SB G$, this induces a group homomorphism
\[
\pi^1\co \Gamma_x\rto G.
\]
We can further compute  the right hand side in \eqref{e-4.18}. It is
\[
(\Gamma_x\times S_x\times G_0\rrto S_x\times G_0)=\Gamma_x\ltimes (S_x\times G_0),
\]
where $\Gamma_x$ acts on $G_0$ via the group homomorphism $\pi^1\co \Gamma_x\rto G$. Hence we prove the following.

\begin{theorem}[The slice theorem]
We have $\CX_U\simeq \Gamma_x\ltimes (S_x\times G)$, where $U\subset |\CX|$ is an open neighborhood of orbit $G\cdot |x|$.
Further, the $G$-action on $\CX_U$ is equivalent to the strict $G$-action on $\Gamma_x\ltimes (S_x\times G)$ by acting on $G$ from right. Then
\[
\SK_{|S_x|}\simeq \Gamma_x\ltimes S_x
\simeq (\Gamma_x\ltimes (S_x\times G))\rtimes G.
\]
\end{theorem}

\subsection{Lie 2-group actions on  orbifolds}\label{Sec-2gp-orbifold}

In this section, we consider the case where the target $\CX$ on which $\SG$ acts is an orbifold, i.e., a proper \'etale Lie groupoid. The differential geometry on orbifolds has many similarity with that on manifolds. It is natural to ask how to formulate $\SG$-equivariant theory of $\CX$ via Kan fibrations.  This subsection is devoted to formulating some fundamental concepts of 2-group $\SG$-equivariant differential geometry on orbifolds.

Let $\CX$ be an orbifold, i.e., a proper \'etale Lie groupoid. The \'etaleness implies the existence of tangent bundle $\ST\CX=(T\CX^1\rrto T\CX^0)$ and cotangent bundle $\ST^*\CX=(T^*\CX^1\rrto T^*\CX^0)$.
Therefore, we have all kinds of tensor bundles over $\CX$. A smooth section of a tensor bundle is called a (smooth) tensor field. A section of $\ST\CX$ is a vector field of $\CX$, let $\Vect(\CX)$  denote the space of vector fields. A section of $\Lambda^k(\ST^\ast \CX)$ is a $k$-form of $\CX$, let $\Omega^k(\CX)$ denote the space of $k$-forms. In general,
we denote the vector space of
 tensor fields of $\mc T\CX$ by $\Gamma(\mc T\CX)$.

 A $k$-form $\omega$ is a pair $(\omega^0,\omega^1)$ of $k$-forms over $\CX^0$ and $\CX^1$, respectively, that are compatible with the source and target maps. Similarly, a vector field
 $V$ is a pair  $(V^0,V^1)$.  As in the differential geometry on manifolds, we have the exterior differential $d$, the contraction $i_V$, namely
\begin{eqnarray*}%
& d\co \Omega^k(\CX)\rto \Omega^{k+1}(\CX),&\qq (\omega^0,\omega^1)\mapsto
(d\omega^0,d\omega^1);\\
& i\co \Vect(\CX)\otimes \Omega^k(\CX)\rto \Omega^{k-1}(\CX),&\qq
(V^0,V^1)\otimes (\omega^0,\omega^1)
\mapsto (i_{V^0}(\omega^0),i_{V^1}(\omega^1)).
\end{eqnarray*}

Now consider a smooth Lie 2-group $\SG$-action on $\CX$ given by the Kan fibration
\[
\pi\co \mc X\rightarrowtail\SK\twoheadrightarrow\SG
\]
We explain how to obtain $\SG$-equivariant objects on $\CX$ via the Kan fibration $\SK$.

\vskip 0.1in
\noindent
{\bf $\SG$-action on the tangent bundle $\ST\CX$.}
\vskip 0.1in
The $\SG$-action on $\CX$ induces an action on the tangent bundle $\ST\CX$. This means we should obtain a Kan fibration
\begin{equation}
\pi_\ST\co \ST\CX \rightarrowtail \SE \twoheadrightarrow \SG.
\end{equation}
Here the $\SE$ is taken to be $\zeta\co \hat\ST\SK\to \SK$, the vertical tangent vector bundle with respect to the fibration $\SK\to \SG$. Precisely, for each
$g\in G_n$, the fiber $E_g$ is the tangent bundle of $K_g=\pi_n^{-1}(g)$; for each point $y\in K_g$, the fiber $\zeta^{-1}(y)$ is the tangent space of $y$ in $K_g$.
The following lemma is straightforward.

\begin{lemma}
$\hat\pi\co \ST\CX\rightarrowtail\hat\ST\SK\twoheadrightarrow \SG$ is a Kan fibration. This defines the induced action of $\SG$ on $\ST\CX$.
\end{lemma}

So we have the Kan fibration
\[
\pi_{\mc T}\co \mc T\mc X\rightarrowtail\hat{\mc T}\SK\twoheadrightarrow\SG
\]
for tensor bundle $\mc T\CX$. For example
\begin{equation}
\pi^k\co \Lambda^k( T^\ast\mc X)\rightarrowtail
\Lambda^k(\hat{ \ST}^\ast\SK)\twoheadrightarrow\SG.
\end{equation}

\vskip 0.1in
\noindent
{\bf $\SG$-equivariant tensor fields on $\CX$.}
\vskip 0.1in
We have already defined $\SG$-equivariant sections of $\SG$-equivariant vector bundles. Hence, $\SG$-equivariant sections of $\mc T\CX\to \CX$ are $\SG$-equivariant tensor fields. When $\mc T\CX=\ST\CX$, a $\SG$-equivariant tensor field is called a {\bf $\SG$-equivariant vector field}. Let $\Vect_\SG(\CX)$ denote the space of $\SG$-equivariant vector fields. Similarly, we have the space $\Omega^k_\SG(\CX)$ of {\bf $\SG$-equivariant $k$-forms}.
The differential $d$ and contraction $i$ can be defined on $\SG$-equivariant objects:
\begin{equation}
d\co \Omega^k_\SG(\CX)\to \Omega_\SG^{k+1}(\CX) \and
i\co \Vect_\SG(\CX)\otimes\Omega_{\SG}^k(\CX)
\to \Omega^{k-1}_\SG(\CX).
\end{equation}

\vskip 0.1in\noindent
{\bf 1-free action of 2-groups}
\vskip 0.1in
Let $\Phi\co \CX\rightarrowtail\SK\twoheadrightarrow\SG$ be the Kan fibration of a $\SG$-action.

\begin{defn}\label{D-1-free-action}
We say the $\SG$-action is {\bf 1-free} if $\SK$ is 2-isotropy free.
\end{defn}

As a consequence of Proposition \ref{P-2-to-1},

\begin{lemma}
The $\SK$ is equivalent to a Lie 1-groupoid if the action is 1-free.
\end{lemma}

\vskip 0.1in
\noindent
{\bf The infinitesimal actions of \'etale Lie 2-group actions on orbifolds}\label{Subsec 3.2}
\def \GL{\mathrm{GL}}
\vskip 0.1in
For the rest of this section, we assume that $\SG$ is \'etale and develop the infinitesimal action.

We first describe the (co)adjoint action of $\SG$ on $\fk g:=T_eG_1$ \footnote{In fact, from the description of the $\SG$-action on $\fk g$ below, one can see that, it is natural to view $\fk g$ as the Lie algebra of $\SG$. However, we will not discuss this issue here as it is not needed in this paper.}, the tangent space of $G_1$ at the unit $e$. Let $\GL(\fk g)$ be the group of linear automorphisms of $\fk g$. It is associated with a $1$-group $\SB\GL(\fk g)$.  The action groupoid with the projection
\[
\pi_\fk g\co \SK(\fk g):= \fk g\rtimes \GL(\fk g)\rto \SB\GL(\fk g)
\]
is the Kan fibration of $\SB\GL(\fk g)$-action on $\fk g$.
Similarly, we have $\GL(\fk g^*)$, $\SB\GL(\fk g^*)$ and $\SK(\fk g^*)$.

The {\bf adjoint action} of $\SG$ on $\fk g$ is given by a strict homomorphism
\begin{equation}
\SAd \co \SG\rto\SB\GL(\fk g)
\end{equation}
described as follows. Since $\SB\GL(\fk g)$ is a 1-group, $\sf Ad$ is determined by $\Ad_1$. Let $g\in G_1$, we describe the $\Ad_1(g)\in \GL(\fk g)$.  Consider an
$\alpha=\s^0_0(g)\in G^{[1]}_1$. Then
\[
d^1_0(\alpha)=d^1_0(\alpha)=e \and
\d^1_0(\alpha)=\d^1_1(\alpha)=g.
\]
Let $U\subseteq G_1$ be a small neighborhood of the unit $e$.  By the \'etaleness of $\SG$, for any $a\in U$, there exists a unique $\alpha_a\in U_{e,g}\subset  G^{[1]}_1$, where $U_{e,g}$ is a small neighborhood of $s^1_0(g)\times_{d^2_1,G_1,d^2_1}s^1_1(g)$, such that
\[
d^1_1(\alpha_a)=a\and \d^1_0(\alpha_a)
=\d^1_1(\alpha_a)=g.
\]
Let $\phi_g\co U\to G_1$ denote the map
\[
a\mapsto d^1_0(\alpha_a).
\]
Set $\Ad_1(g)$ to be the tangent map of $\phi_g$ at $e$.  It is direct to check that for a $2$-simplex $g_{012}\in G_2$, there holds
\[
\Ad_1(g_{01})\cdot\Ad_1(g_{12})=\Ad_1(g_{02}).
\]
Therefore, the $\Ad_1$ above determines a strict morphism $\msf{Ad}\co\SG\rto\SB\GL(\fk g)$.

Let
\begin{equation}
\g_{\msf{Ad}}:=\msf{Ad}^\ast \SK(\fk g)=\SG\times_{\SB\GL(\fk g)}\SK(\fk g)
\end{equation}
be Kan fibration with respect to adjoint action of $\SG$ on $\fk g$. Note that $(\g_\SAd)_n=\fk g\times G_n$.

Similarly,  the {\bf coadjoint action} is
\[
\SAd\co \SG\rto\SB\GL(\fk g^*),
\]
which over $G_1$ is given by
\begin{align}\label{E Ad-*}
g\mapsto (\Ad_1(g)^*)\inv.
\end{align}
The Kan fibration for this action is
\[
\g^*_{\SAd}:=\SAd^*\SK(\fk g^*)=\SG\times_{\SB\GL(\fk g^*)}\SK(\fk g^*)\rto \SG.
\]
Then $(\g_\SAd^*)_n=\fk g^*\times G_n$.

We next describe the infinitesimal action of a $\SG$-action on $\CX$. This should have two properties:
\begin{enumerate}
\item assign to each $\xi\in \fk g$ a vector field $V_\xi$ on $\CX$, i.e., give a map $\phi\co\fk g\rto \Vect(\CX)$;
\item $\phi$ is $\SG$-equivariant.
\end{enumerate}
We realize $\phi$ by a section $\sigma$ of the vector bundle
\begin{equation}\label{e-inf-vector-bundle}
\pi_2^\ast \ST\CX\to \fk g\times \CX
\end{equation}
where $\pi_2\co \fk g\times \CX\to \CX$ is the projection. Of course, $\pi_2^\ast \ST\CX=\fk g\times \ST\CX$.

This bundle and its base admit a (diagonal) $\SG$-action. The bundle is  a $\SG$-equivariant bundle. Hence, we may require $\sigma$ to be a $\SG$-equivariant section which ensures that $\phi$ is $\SG$-equivariant.

\begin{remark}
Let $\pi_i\co \SX_i\rightarrowtail \SK_i\twoheadrightarrow \SG, i=1,2,$ be two Kan fibration. Then the diagonal $\SG$ action on $\SX_1\times \SX_2$ is given by
\begin{equation}
(\pi_{1},\pi_2)\co
\SX_1\times\SX_2\rightarrowtail
\SK_1\times_\SG\SK_2\twoheadrightarrow
\SG.
\end{equation}
\end{remark}

The Kan fibration version of \eqref{e-inf-vector-bundle} is the bundle
\begin{align}
\g_\SAd\times_\SG {\sf \hat T}\SK
\rto \g_\SAd\times_\SG \SK.
\end{align}
Let $\sigma_\SK$ be the Kan-section of $\sigma$.
We construct $\sigma_\SK$ which serves as the infinitesimal action.

An important observation is that $\sigma_\SK$ is determined by its restriction on
\[
(\g_\SAd\times_\SG\SK)_0=\fk g\times X^0.
\]
In fact, for any $g\in G_n$, $n\geqslant 1$, the
$d^n_j\co K_g\rto K_{g'}$, where $g'=d^n_j(g)$, is a local diffeomorphism, the section at $K_g$ is determined by $K_g'$.
Denote this restriction vector field by $\sigma_0$. To define $\sigma_0$ we need the following observation.

\begin{lemma}\label{L etale-kan-fib}
Since both $\CX$ and $\SG$ are \'etale, the Lie $2$-groupoid Kan fibration $\pi\co\SK\rto\SG$ is an \'etale Kan fibration.
\end{lemma}
\begin{proof}
This follows by a simple dimension counting. We omit the details.
\end{proof}

Now we proceed to define $\sigma_0$. For any pair $(\xi,x)\in \fk g\times X^0$, by the \'etale condition $\Kan!!(1,1)(\pi)$, the \'etale map $\tau_{1,1}\co K_{1}\rto K_{0}\times G_1=X^0\times G_1$, there is a unique tangent vector $\xi'\in T_{s^0_0(x)} K_{1}$ satisfying
\[
d_{s^0_0(x)}\tau_{1,1}(\xi')=(0,\xi)\in T_{(x,e)}(X^0\times G_1).
\]
Then we set
\begin{align}\label{E V-0-1}
\sigma_0(\xi,x):=d_{s^0_0(x)}d^1_0(\xi').
\end{align}

\begin{remark}
The $\sigma_0$ can be also defined as follows. For every $(\xi,x)$, take a curve $\gamma_\xi(t)$ in $G_1$ representing $\xi\in T_eG_1$, i.e., $\gamma_\xi(0)=e$, and $\gamma_\xi'(0)=\xi$. Then by $\Kan!!(1,1)(\pi)$ we have a curve $\tilde\gamma_\xi(t)$ in $K_1$ satisfying $\tilde\gamma_\xi(0)=s^0_0(x)$ and $\tau_{1,1}(\tilde\gamma_\xi(t))=(x,\gamma_\xi(t))\in X^0\times G_1$, therefore $d^1_1(\tilde\gamma_\xi(t))=x$ and $\pi_\rho(\tilde\gamma_\xi(t))=\gamma_\xi(t)$. Then we have
\begin{align}\label{E V-0-2}
\sigma_0(\xi,x)=d_{s^0_0(x)}d^1_0(\tilde\gamma_\xi'(0)) =(d^1_0\circ\tilde\gamma_\xi(t))'|_{t=0}.
\end{align}
\end{remark}

\begin{remark}\label{R-inf-actionn}
The section $\sigma\co \fk g\times \CX\rto \fk g\times \ST\CX$ can be viewed as a bundle map, still denoted by
\begin{equation}
\sigma\co \fk g\times \CX\rto \ST\CX.
\end{equation}
Then $\sigma_\SK$ can be viewed as
a Kan-bundle map
\begin{equation}
\sigma_\SK\co \g_\SAd\times_\SG \SK\rto \hat\ST\SK.
\end{equation}
\end{remark}

\section{Symplectic orbifolds and symplectic reductions}\label{Sec-symplectic-reduction}

As an application, we formulate Hamiltonian 2-group actions on symplectic orbifolds and study their reductions. In \cite{Hoffman-Sjamaar2021-IMRN}, a similar issue was studied. Our approach is essentially different; see Remark \ref{R-comparison} for the comparison.

\subsection{Statements of results}
\label{Sec-sym-reduction-statement}
Let us first recall the theory on symplectic manifolds (cf. cf. \cite{Marsden-Weinstein1974-reduction,Meyer1973-reduction}).
Let $(X,\omega)$ be a symplectic manifold, $G$ be a compact Lie group and $\fk g$ the Lie algebra of $G$. Suppose $G$ acts smoothly on $X$. For each $\xi\in \fk g$, let $V_\xi$ be the infinitesimal vector field generated by $\xi$. The action is called a Hamiltonian action if there exists an equivariant map $\mu\co X\rto \fk g^\ast$ such that, for any $\xi\in \fk g$,
\[
d\langle \mu, \xi\rangle=-i(V_\xi, \omega).
\]
The $\mu$ is called a moment map of the Hamiltonian action.

Now suppose $0\in \fk g^\ast$ is a regular value of $\mu$, then the level set $Y=\mu^{-1}(0)$ is a smooth submanifold of $X$. For any $y\in Y$, $V_\xi(y)\in T_yY$, and
\[
V_\xi(y)=0\iff \xi=0.
\]
This implies that $G$ acts on $Y$ and the action is locally free. If the action is {free}, the symplectic reduction is the quotient manifold $Z=Y/G$. Moreover, $\omega$ induces a symplectic structure $\omega_0$ on $Z$, $(Z,\omega_0)$ is known as a symplectic reduction.

Now let us focus on the case where the action is {\em not free}. We obtain a quotient groupoid $\tilde{\mc Z}=Y\rtimes G$. Since the action is locally free, $\tilde {\mc Z}$ is equivalent to a proper, \'etale Lie groupoid $\mc Z$, which also carries a symplectic structure $\omega_0$. Then we get the symplectic reduction $(\mc Z,\omega_0)$.

Using the language of Lie groupoids, we treat $X$ as a Lie 0-groupoid and it admits a Lie 1-group $\SB G$-action. The symplectic reduction $Y\rtimes \SB G$ is a Lie 1-groupoid, and is equivalent to a proper, \'etale Lie groupoid $\mc Z$.
Furthermore, if the $\SB G$-action on $Y$ is free, $\mc Z$ can be further equivalent to a Lie 0-groupoid. This is illustrated as the following line:
\[
\{\mbox{Lie $0$-groupoid}\}
\xrightarrow{\mbox{reduction}}
\{\mbox{Lie $1$-groupoid}\}
\xrightarrow{\mbox{free action}}
\{\mbox{Lie 0-groupoid}\}.
\]
From this point of view, we see that the symplectic reduction of (a Hamiltonian {\bf 1-group} action on) a Lie 0-groupoid lies outside of the scope of Lie 0-groupoids unless the action satisfies a further freeness assumption.

In this section, our main result Theorem \ref{T-orbifold-reduction} may be summarized as the following:
\begin{enumerate}
\item The symplectic reduction of a Hamiltonian 1-group action on an orbifold is an orbifold, i.e,
\[
\{\mbox{Lie $1$-groupoid}\}
\xrightarrow{\mbox{1-group reduction}}
\{\mbox{Lie $1$-groupoid}\}.
\]
\item The symplectic reduction of a Hamiltonian {\bf \'etale 2-group} action on an orbifold is an \'etale Lie 2-groupoid:
\[
\{\mbox{Lie $1$-groupoid}\}
\xrightarrow{\mbox{2-group reduction}}
\{\mbox{Lie $2$-groupoid}\}.
\]
\item Continue (2), if the reduction groupoid is 1-free, it is equivalent to an orbifold:
\[
\{\mbox{Lie $2$-groupoid}\}
\xrightarrow{\mbox{1-free action}}
\{\mbox{Lie 1-groupoid}\}.
\]
\end{enumerate}

In \cite{Hoffman-Sjamaar2021-IMRN}, Hoffman--Sjamaar also studied the Hamiltonian 2-group $\SG$-action \footnote{Their definition is different from the one in Definition \ref{D hamiltonian-action}.} on symplectic Lie groupoids. For reader's convenience, we make the following comparison.

\begin{remark}\label{R-comparison}
[{\bf Comparison with the work in \cite{Hoffman-Sjamaar2021-IMRN}}]
~

{\em On results.} In \cite{Hoffman-Sjamaar2021-IMRN}, the target Lie groupoid $\CX$ is assumed to be \'etale, and results (1) and (3) listed above are essentially obtained.

We mainly focus on the case that $\CX$ is an orbifold, i.e., proper and \'etale, though the properness assumption is not essential. On the other hand, we give a complete answer for the symplectic reduction by including (2).

{\em On approaches.} In \cite{Hoffman-Sjamaar2021-IMRN}, the Hamiltonian action and symplectic reduction are constructed under the assumption of a strict 2-group with strict action. For a non-strict 2-group $\SG$
with a (non-strict) action on $\CX$, it is theoretically proved that  one may replace $\SG$ by a strict 2-group $\tilde \SG$, $\CX$ by $\tilde{\CX}$
such that the original action is equivalent to a strict action of $\tilde \SG$ on $\tilde{\CX}$. Also note that  $\tilde{\CX}$ can not be \'etale in general.
\end{remark}

The approach in \cite{Hoffman-Sjamaar2021-IMRN} is certainly not convenient for general actions.
In this section, we provide a direct approach to study the symplectic geometry of any action via Kan fibrations. An additional advantage of our approach is that the \'etaleness is preserved.

\subsection{Hamiltonian actions on symplectic orbifolds }\label{Sec-Hamiltonian}

In this subsection we define Hamiltonian actions of \'etale Lie $2$-groups on symplectic orbifolds. We first recall the definition of symplectic forms on orbifold groupoids.

\begin{defn}\label{D-symplectic}
Let $\CX$ be an orbifold groupoid. A 2-form $\omega=(\omega^0,\omega^1)$ on $\CX$ is called a symplectic form  if $\omega^0$ on $\CX^0$ is a symplectic form. Then $(\CX,\omega)$ is called a {\bf symplectic orbifold groupoid}.
\end{defn}

This definition also works for general \'etale Lie groupoids. By the \'etaleness of $\CX$, for a symplectic form $\omega$, the 2-form $\omega^1$ on $\CX^1$ is also a symplectic form.

Let $(\CX,\omega)$ be a symplectic orbifold groupoid and $\SG$ be an \'etale Lie 2-group, and $\pi\co \CX\rightarrowtail\SK\twoheadrightarrow \SG$ be a Kan fibration.

\begin{defn}\label{D symplectic-action}
We say the action given by $\SK$ is a {\bf symplectic action} if $\omega$ is a $\SG$-equivariant form. Let $\omega_\SK$ be the corresponding Kan-form on $\SK$.
\end{defn}

Let $\sigma$ and $\sigma_\SK$
be the bundle map of the infinitesimal action (cf. Remark \ref{R-inf-actionn}). Let also
$\sigma_\SK$ be the Kan-vector field of the infinitesimal action. Suppose its restriction on $\fk g\times \CX$ is $V$.
We have the  homomorphism
\[
\zeta\co \fk g\times \CX\xrightarrow{\sigma}
\ST\CX\xrightarrow{-i(\cdot,\omega)} \ST^\ast\CX,
\]
which is a composition of vector bundle maps.
Its equivariant version is
\[
\zeta_\SK\co \g_{\SAd}\times_\SG \SK\xrightarrow{\sigma_\SK}
\hat\ST\SK\xrightarrow{-i(\cdot,\omega_\SK)} \hat\ST^\ast\SK.
\]
The dual of $\zeta$ is a bundle map
\[
\tilde\zeta^\ast\co \ST\CX\rto \fk g^\ast\times\CX,
\]
and its equivariant version is
\[
\tilde\zeta^\ast_\SK\co \hat \ST\SK\rto \g^\ast_\SAd \times_\SG\SK.
\]
With further projections, we have
\[
\zeta^\ast\co \ST\CX\rto \fk g^\ast
\and
\zeta^\ast_\SK\co \hat \ST\SK\rto \g^\ast_\SAd.
\]

\begin{defn}\label{D hamiltonian-action}
Let $\pi\co \CX\rightarrowtail\SK\twoheadrightarrow\SG$ be the Lie $2$-groupoid Kan fibration of a symplectic action of an \'etale Lie $2$-group $\SG$ on a symplectic orbifold groupoid $(\CX,\omega)$.    If there is an $\SG$-equivariant  strict homomorphism $\mu\co \CX\to \fk g^\ast$ (equivalently, a Kan homomorphism $\mu_\SK\co \SK\to \g^\ast_\SAd$) such that $d\mu=\zeta^\ast$, we call $\SK$ a {\bf Hamiltonian action} with the moment map $\mu$.
\end{defn}

By the definition, we have
\begin{equation}\label{E-momoent-cond}
d\mu=i_V(\omega)\and
\hat d\mu_\SK=i_{V_\SK}(\omega_\SK).
\end{equation}
Here $\hat d$ means the differential followed by the projection to $\hat\ST^*\SK$. The \eqref{E-momoent-cond} is the classical condition on moment maps (cf. 
\cite[\S 22.1]{daSilva2008}).

\subsection{Symplectic reduction}
\label{Sec-sym-reduction}
Now we discuss the symplectic reductions with respect to Hamiltonian actions. The arguments are essentially same as that for smooth manifolds.

For the $0\in \fk g^\ast$,
denote $\mu^{-1}(0)$ by $\CY$. As $\mu$ is $\SG$-equivariant, the  action preserves $\mc Y$. The Kan fibration $\SH\rto \SG$ of the $\SG$-action on $\CY$ is given by
\[
\SH=\mu_\SK^{-1}(0\times\SG), \qq\text{where}\qq 0\times\SG\subseteq \g_\SAd^\ast.
\]
So $\SH\subseteq\SK$ is closed.

\begin{lemma}
If $0\in\fk g^*$ is a regular value of $\mu$, the level set $\mc Y$ is an orbifold groupoid. Moreover $\sf H$ is an \'etale Kan fibration of Lie $2$-groupoids over $\SG$ with fiber $\mc Y$.
\end{lemma}
\begin{proof}
When $0\in\fk g^*$ is a regular value of $\mu$, $\mc Y$ is a full closed subgroupoid of $\CX$, hence is an orbifold groupoid. The rest is obvious.
\end{proof}

Since $\SH\rto \SG$ is the Lie $2$-groupoid Kan fibration of the $\SG$-action on $\CY$, we have

\begin{lemma}
The restriction of $\sigma_{\SK}$ on $\g_\SAd\times_\SG \SH$ is a section of $\g_\SAd\times_\SG \hat\ST\SH\rto \g_\SAd\times_\SG\SH$. Denote the restriction section by $\sigma'_\SH$ and its restriction on $\fk g\times \mc Y$ by $\sigma'$.
\end{lemma}

\begin{lemma}\label{L V-0-nonvanishing}
If $0\in\fk g^*$ is a regular value of $\mu$, for any $\xi\in \fk g$, the vector field $\sigma_0(\xi,\cdot)$ is nowhere vanishing on $Y^0$.
\end{lemma}
The proof is standard, we omit it.

As explained in Remark \ref{R-inf-actionn}, $\sigma'$ and $\sigma'_\SH$ are bundle maps
\[
\sigma'\co \fk g\times \mc Y\rto \ST\mc Y\and
\sigma'_\SH\co \g_\SAd\times_\SG\SH\rto \hat\ST\SH.
\]
By Lemma \ref{L V-0-nonvanishing}, $\sigma'$ is an injective map and is determined by the restriction of $\sigma_0$, denoted by $\sigma'_0$. Let $P\subseteq TY^0$ be the sub-bundle of the image of $\sigma'_0$; then $P$ is an integrable distribution over $Y^0$. We have the exact sequence of vector bundles over $Y^0$
\[
P\rto TY^0\rto Q,
\]
where $Q$ is the quotient bundle $TY^0/P$. One can still define the $\omega$-complementary $P^\omega$ of $P$ and we have $P^\omega=TY^0$. Hence $\omega$ induces a non-degenerated section $\omega'$ of $\Lambda^2(Q^\ast)$.

The above argument certainly can be applied to the equivariant setting, so we have the sub-bundle $\sf P$ of ${\sf \hat T} \SH$, the quotient bundle ${\sf Q}$ over $\sf H$, and the $\omega'_\SH\in\Lambda^2(\sf Q^\ast)$ that restricts to $\omega'$.

To summarize, we have

\begin{theorem}\label{T reduction}
Let $\mu\co\CX\rto \fk g^\ast$ be the moment map of the Hamiltonian action. If $0\in\fk g^*$ is a regular value of $\mu$, then $\CY=\mu^{-1}(0)$ is an orbifold that admits the induced $\SG$-action with the Kan fibration $\SH\subseteq \SG$. The symplectic form $\omega$ on $\CX$ induces a section $\omega'$ on $\Lambda^2(Q^\ast)\rto \CY$ that is $\SG$-equivariant. Hence, we have a section $\omega'_\SH$ of the vector bundle
$
\Lambda^2(\sf Q^\ast).
$
\end{theorem}

\begin{defn}\label{D reduction}
We call $(\SH,\omega'_{\SH})$ obtained in Theorem \ref{T reduction} the {\bf symplectic reduction} of
 the Hamiltonian action $\SK$ at $0\in\fk g^\ast$.
\end{defn}

We next explain that $\SH$ has an equivalent \'etale Lie 2-groupoid.

\begin{theorem}\label{T etale-reduction}
Under the assumption of Theorem \ref{T reduction}, the symplectic reduction $(\SH,\omega'_{\SH})$
is equivalent to a symplectic \'etale Lie 2-groupoid $(\SZ,\omega'')$. Here by symplectic we mean that $\omega''_n$ on $Z_n$ is symplectic for every $n\geqslant 0$.
\end{theorem}
\begin{proof} We first  construct $Z_0$ and a smooth map $q\co Z_0\rto H_0$.

Recall that $H_0=Y^0$. The integrable constant rank distribution $P$ over $H_0$ integrates into a regular foliation $\fk F_0$ on $H_0$, with leaf, denoted by $O_x$, passing through $x$. Now, take a local slice $S_x$ of this foliation, so that the following holds
\[
T_yO_y\oplus T_yS_x=P_y\oplus T_yS_x=T_yH_0,\qq \forall\qq y\in S_x,
\]
i.e., $S_x$ is a slice of $\fk F$ at each $y\in S_x$. We will show that the following map is a local diffeomorphism
\begin{align}\label{E map-d-1-0}
d^1_0\co (d^1_1)^{-1}(S_x)\rto F_0.
\end{align}

By the $\Kan!!(1,1)(\pi)$ of the \'etale Lie 2-groupoid Kan fibration $\pi\co\SH\rto\SG$, we have a local diffeomorphism
\[
\tau_{1,1}=(d_1^1,\pi_1)\co H_1\rto H_0\times G_1.
\]
Take a point $x_1\in (d^1_1)\inv(x)$ with $\tau_{1,1}(x_1)=(x,g)$ and $y=d^1_0(x_1)\in O_x=O_y$, we have the linear isomorphism
\[
d_{x_1}\tau_{1,1}\co T_{x_1}H_1\xrightarrow{\cong}T_xH_0\oplus T_gG_1=P_x\oplus T_xS_x\oplus T_gG_1
\]
and $T_{x_1}(d^1_1)\inv(S_x)\cong T_xS_x\oplus T_gG_1$. Then to show that $d^1_0$ in \eqref{E map-d-1-0} is a local diffeomorphism, we only need to show that the linear map $f:= d_{x_1}(d^1_0)\circ (d_{x_1}\tau_{1,1})\inv\co P_x\oplus T_xS_x\oplus T_gG_1\rto T_yH_0$ restricts to an isomorphism $T_xS_x\oplus T_gG_1\rto T_yH_0$.

By the definition of $\sigma_0(\xi,x)$ and Lemma \ref{L V-0-nonvanishing}, the $f$ restricts to
\[
f\co T_gG_1\xrightarrow{\cong} T_yO_x=T_yO_y=P_y.
\]
Since the $\sigma_\SK$ is compatible with the face maps, the $f$ also restricts to
\[
f\co P_x\xrightarrow{\cong} P_y.
\]
Since $d^1_0\co H_1\rto H_0$ is a submersion, the linear map $f$ is surjective. Therefore,
\[
f(T_xS_x)\oplus P_y=T_yH_0.
\]
This show that $d^1_0$ is a local diffeomorphism around each point in $(d^1_1)\inv(x)$.

For another $x_1'\in (d^1_1)\inv(S_x)$, suppose $d^1_1(x_1')=x'\in S_x$, then since $S_x$ is also a slice of $\fk F$ at $x'$, the above argument also shows that $d^1_0$ is a local diffeomorphism around $x_1'$. This finishes the proof that the map \eqref{E map-d-1-0} is a local diffeomorphism.

Now choose a collection of points $\{x_i\}_{i\in J}$ such that $\{|S_{x_i}|\}_{i\in J}$ covers the coarse space $|Y^0|$ of $\mc Y$. Then we take $Z_0$ to be the disjoint union of $S_{x_i}$ and $q\co Z_0\rto H_0=Y^0$ the inclusion on each component. Since $\{|S_{x_i}|\}_{i\in J}$ covers $|Y^0|$, the map $d^1_0\circ pr_1\co Z_0\times_{q,F_0,d^1_1}H_1\rto H_0$ is surjective; it is also a local diffeomorphism, since the map \eqref{E map-d-1-0} is a local diffeomorphism. It is clear that
\[
\psi\co Z_0\times_{q,H_0,d^1_1} H_1\xrightarrow{pr_2}H_1\xrightarrow{d^1_0}H_0
\]
is a surjective submersion. Now
the Lie 2-groupoid structure of $\SH$ induces a Lie 2-groupoid structure on $Z_0$ via $q$. Denote the Lie 2-groupoid by $\SZ$.
By dimension counting, the $\SZ$ satisfies $\Kan!!(n)$ for $n=1,2$, i.e., is \'etale.
Finally, since we take $S_x$ as the slice, the tangent bundle $\msf T\SZ$ is isomorphic to $q^\ast \msf Q$. Hence $\omega'$ induces a symplectic structure $\omega''$ on $\SZ$.
\end{proof}

\begin{coro}\label{C 1-grpd-action}
Under the assumption of Theorem \ref{T reduction}, if furthermore $\SG$ is an \'etale Lie $1$-group, then the symplectic reduction $(\SH,\omega_\SH')$ is a Lie $1$-groupoid, and the symplectic reduction $(\SZ,\omega_\rho'')$ is an \'etale Lie $1$-groupoid. \end{coro}
\begin{proof}
When $\SG$ is a Lie $1$-groupoid, the $\Kan!(2)(\pi)$ for $\pi\co \SH\rto\SG$ and $\Kan!(2)$ of $\SG$ implies the $\Kan!(2)$ of $\SH$. Therefore $\SH$ is a Lie $1$-groupoid, and so is the induced groupoid $\SZ$.
\end{proof}

\begin{remark}\label{R_etalereduction}
Up to this point, we do not use any property of properness of $\CX$. We have completed the general construction of the symplectic reduction of Hamiltonian actions of \'etale Lie $2$-groups on symplectic orbifold groupoids, which are (\'etale) Lie $2$-groupoids in general.
\end{remark}

\begin{prop}\label{P proper-reduction}
Under the assumption of Theorem \ref{T reduction}, if furthermore the action, i.e., $\SK$, is proper, then the reduction $\SH$ and $\SZ$ are both proper.
\end{prop}
\begin{proof}
By definition, $\SK$ is a proper Lie $2$-groupoid. Since $\SH$ is closed in $\SK$, $\SH$ is also proper. Hence, $\SZ$ is also proper.
\end{proof}

Moreover, when $\SH$ is proper and $2$-isotropy free (equivalently, $\SG$-action on $\mc Y$ is 1-free), even if $\SG$ is a Lie 2-group, the reduction can still  be equivalent to a Lie 1-groupoid.

\begin{prop}\label{P 1-grpd-reduction}
Under the assumption of Theorem \ref{T reduction}, if furthermore $\SH$ is $2$-isotropy free and proper, then the symplectic reduction $(\SH,\omega'_\SH)$ and also $(\SZ,\omega'')$ are equivalent to a proper, \'etale Lie $1$-groupoid $\tilde\SZ$ with an induced symplectic structure $\tilde\omega''$.
\end{prop}
\begin{proof}
By Proposition \ref{P proper-reduction},  properness of the action implies that the reductions $\SH$ and $\SZ$ are both proper Lie $2$-groupoids. Since $\SH$ is 2-isotropy free,
so is $\SZ$. By Proposition \ref{P-2-to-1}, we could construct a proper \'etale Lie $1$-groupoid $\tilde\SZ$ that is equivalent to $\SZ$. Moreover, by the construction of $\tilde \SZ$ in Proposition \ref{P-2-to-1}, and the fact that $\omega''$ is a symplectic form over $\SZ$, it induces a symplectic form $\tilde\omega''$ on $\tilde\SZ$.
\end{proof}

Summarizing all results and, with properness assumption, we have
\begin{theorem}\label{T-orbifold-reduction}
Let $(\CX,\omega)$ be a symplectic orbifold groupoid, $\SG$ be an \'etale Lie 2-group.
\begin{enumerate}
\item Let
$\Phi\co \CX\rightarrowtail\SK\twoheadrightarrow \SG$ be the Lie $2$-groupoid Kan fibration of a Hamiltonian $\SG$-action on $\CX$.
\item Let $\mu_\SK\co \hat \ST\SK\to \g^\ast_\SAd$ be the moment map.
\item Suppose the action is proper, i.e., $\SK$ is proper.
\end{enumerate}
Then the symplectic reduction $(\SZ,\omega')$ is a proper, \'etale symplectic Lie $2$-groupoid.

The pair
$(\SZ,\omega')$ is (equivalent to) a proper, \'etale symplectic Lie $1$-groupoid, i.e., a symplectic orbifold groupoid if one of the following holds:
\begin{enumerate}
\item[(1)] $\SG$ is a Lie 1-group;
\item[(2)] $\SH$ is 2-isotropy free.
\end{enumerate}
\end{theorem}
\begin{proof}
By theorem \ref{T etale-reduction} and Proposition \ref{P proper-reduction}, the $(\SZ,\omega')$ is a proper, \'etale symplectic Lie $2$-groupoid. When $\SG$ is a Lie $1$-groupoid, by Corollary \ref{C 1-grpd-action}, $(\SZ,\omega')$ is a proper, \'etale symplectic Lie $1$-groupoid. When $\SH$ is 2-isotropy free, by Proposition \ref{P 1-grpd-reduction}, the $(\SZ,\omega')$ is equivalent to a proper, \'etale symplectic Lie $1$-groupoid.
\end{proof}

\bibliographystyle{abbrv}

\bibliography{chengyong}
\end{document}